\documentclass [11pt]{amsart}
\usepackage{amscd,amsmath,amsthm,amssymb,amsfonts,epsfig}
\usepackage{bbm}
\usepackage{graphicx}
\usepackage[utf8]{inputenc}
\usepackage{xcolor}
\usepackage{mathrsfs}  
\usepackage{esvect}
\usepackage{enumitem}
\usepackage{indentfirst}
\usepackage{caption}
\usepackage{subcaption}
\usepackage[export]{adjustbox}
\usepackage{tkz-euclide}
\usepackage{tcolorbox}
\usepackage{pgfplots}
\pgfplotsset{compat=newest}
\usepgfplotslibrary{fillbetween}
\usepackage{comment}
\usepackage{url}
\usepackage{hyperref}

\usepackage[a4paper,top=3cm,bottom=2cm,left=3cm,right=3cm,marginparwidth=1.75cm]{geometry}

\usetikzlibrary{intersections,positioning, calc}

\renewcommand{\bar}{\overline}
\newcommand{\R}{\mathbb{R}}

\newcommand{\pa}{\partial}

\newcommand{\be}{\mathbf{e}}
\newcommand{\bn}{\mathbf{n}}

\newcommand{\bt}{\mathbf{t}}

\newcommand{\cE}{\mathcal{E}}

\newcommand{\cL}{\mathcal{L}}
\newcommand{\cM}{\mathcal{M}}

\newcommand{\cD}{\mathcal{D}}
\newcommand{\cH}{\mathcal{H}}
\newcommand{\cG}{\mathcal{G}}
\newcommand{\cT}{\mathcal{T}}
\newcommand{\ent}{\mathrm{Ent}}

\newcommand{\vare}{\varepsilon}

\newtheorem{theorem}{Theorem}[section]
\newtheorem{corollary}[theorem]{Corollary}
\newtheorem{lemma}[theorem]{Lemma}
\newtheorem{proposition}[theorem]{Proposition}

\theoremstyle{remark}
\newtheorem{remark}[theorem]{Remark}

\theoremstyle{definition}

\title{Classification of ancient finite-entropy curve shortening flows}

\author{Kyeongsu Choi}
\address{KC: School of Mathematics, Korea Institute for Advanced Study, 85 Hoegiro, Dongdaemun-gu, Seoul 02455, Republic of Korea}
\email{choiks@kias.re.kr}

\author{Dong-Hwi Seo}
\address{DS: Institute of Mathematics (IMAG), University of Granada, C. Ventanilla, 11, Centro, 18001 Granada, Spain}
\email{donghwi.seo26@gmail.com}

\author{Wei-Bo Su}
\address{WS: Department of Mathematics, National Central University, No. 300, Zhongda Rd., Zhongli District, Taoyuan City 320317, Taiwan}
\email{weibosu@math.ncu.edu.tw}

\author{Kai-Wei Zhao}
\address{KZ: Department of Mathematics, UC Irvine, Irvine, California 92697 USA}
\email{kaiweiz2@uci.edu}

\begin{document}

\begin{abstract}
We prove that any ancient smooth embedded finite-entropy curve shortening flow is one of the following: a static line, a shrinking circle, a paper clip, a translating grim reaper, or a graphical ancient trombone. 

An ancient trombone is an immersed ancient flow, either compact or non-compact, obtained by gluing together $m$ translating grim reaper curves. For each $m$, there exists a $(2m-1)$-parameter family of graphical ancient trombones, up to rigid motions and time shifts as constructed by Angenent-You.  

In particular, our result implies that any compact ancient smooth embedded finite-entropy flow is convex. Moreover, any non-compact ancient smooth embedded finite-entropy flow is either a static line or a complete graph over a fixed open interval.
\end{abstract}

\maketitle 
\section{Introduction}

Let $\gamma: \mathbb{L}\times I\to \mathbb{R}^2$, where $\mathbb{L}=\mathbb{R}^1$ or $\mathbb{S}^1$, be a one-parameter family of smooth immersions satisfying 
\begin{equation}
    \gamma_{t} = \boldsymbol{\kappa},
\end{equation}
where $\boldsymbol{\kappa}$ is the curvature vector $\gamma_{ss}$ of $\gamma$. Then, we say that the family $\mathcal{M}=\cup_{t\in I}(M_t,t)$ of the planar curves $M_t=\gamma(\mathbb{L},t)$ is a solution to the curve shortening flow. In particular, if the time interval is $I=(-\infty, T)$, then the flow is called ancient.
Since the flow is governed by a parabolic equation, ancient solutions can be classified by Liouville-type theorems. Daskalopoulos–Hamilton–Sesum \cite{DHS10} showed that any ancient, convex, compact solution to the curve shortening flow is either a shrinking circle or a paper clip. Later, Bourni–Langford–Tinaglia \cite{BLT20} classified all ancient convex flows by proving that any ancient, convex, complete, noncompact solution to the curve shortening flow is a translating grim reaper. In the previous paper \cite{CSSZ24}, the authors further extended these results to ancient low-entropy flows, showing that any ancient smooth flow embedded in $\mathbb{R}^2$ with the entropy $\text{Ent}[\mathcal{M}]<3$ is either convex or a static line. Here, $\text{Ent}[M_t]$ denotes the entropy of the curve $M_t$, first introduced by Magni--Mantegazza \cite[Definition 1.9]{MM09} and later extensively studied by Colding--Minicozzi \cite{CM12}, and $\text{Ent}[\mathcal{M}]$ denotes
\begin{equation}
\text{Ent}[\mathcal{M}]:=\sup_{t\in I}\text{Ent}[M_t].
\end{equation}
See \eqref{def:entropy} for the full definition.
In this paper, we weaken the low-entropy assumption to finite-entropy $\text{Ent}[\mathcal{M}]<+\infty$. Note that there are still a lot of ancient curve shortening flows embedded in $\mathbb{R}^2$ with infinite entropy; see Halldorsson \cite{Ha12}, Charyyev \cite{Ch22}, Zhang-Olson-Khan-Angenent \cite{AIOZ23}.

\begin{theorem}\label{thm:main}
     An ancient smooth curve shortening flow embedded in $\mathbb{R}^2$ with finite-entropy is one of the following: a shrinking circle, a paper clip, a static line, a translating grim reaper, or a graphical ancient trombone.
\end{theorem}

Angenent-You \cite{AY21} constructed ancient trombones by gluing translating grim reapers. They are immersed flows with finite total curvature. Namely, ancient trombones satisfy
\begin{equation}
\sup_{t\in I}  \int_{M_t} |\boldsymbol{\kappa}| \:ds  <+\infty.
\end{equation}
Indeed, in  \cite{SZ24} the third and last named authors recently showed that an ancient flow with finite total curvature has finite-entropy. Hence, the finite-entropy assumption can be replaced by that of finite total curvature.

\begin{corollary}\label{cor:main1}
An ancient smooth curve shortening flow embedded in $\mathbb{R}^2$ with finite total curvature is one of the following: a shrinking circle, a paper clip, a static line, a translating grim reaper, or a graphical ancient trombone.
\end{corollary}

Also, by the classification theorem \ref{thm:main}, we can extend the result \cite{DHS10}
 by Daskalopoulos–Hamilton–Sesum from convex flows to finite-entropy flows.
 
\begin{corollary}\label{cor:main2}
Let $\mathcal{M}$ be an ancient smooth closed curve shortening flow embedded in $\mathbb{R}^2$ with finite-entropy. Then, it is convex.
\end{corollary}

Furthermore, we can observe that nonstatic noncompact finite-entropy flows are graphs.

\begin{corollary}\label{cor:main3}
Let $\mathcal{M}$ be a noncompact, complete, ancient, smooth solution to the curve shortening flow in $\mathbb{R}^2$ with finite-entropy, which is not a static line, and is then necessarily a complete graph over a bounded open interval.
\end{corollary}

To be specific, after a suitable rotation, an ancient graphical trombone converges to $(m+1)$ lines $\{x_2=a_0\}, \cdots, \{x_2=a_m\}$, where $m\geq 2$ and $a_0<a_1< \cdots <a_m$. Moreover, it remains as a complete graph over the open interval $(a_0,a_m)$ for all $t\in (-\infty,+\infty)$. Furthermore, as discussed in \cite{AY21}, the sub-flow $M_t \subset \{a_i<x_2<a_{i+1}\}$ converges to a translating grim repaer curve whose tip is located at 
\begin{equation}
(x_1,x_2)= \left(b_i+\frac{(-1)^i\pi }{a_i-a_{i-1}}t,\frac{a_i-a_{i-1}}{2}\right),
\end{equation}
after a possible reflection. See Figure~\ref{fig: trombone} for an illustration. Hence, Angenent-You \cite{AY21} constructed $(2m-1)$-parameter family of graphical ancient trombones, which are determined by $a_0,\cdots,a_m$ and $b_1,\cdots,b_m$. Therefore, our result \ref{thm:main} implies the following.

\begin{corollary}\label{cor:main4}
Let $\mathcal{M}$ be an ancient smooth curve shortening flow embedded in $\mathbb{R}^2$ with $\text{Ent}[\mathcal{M}]=m+1>2$. Then, $\mathcal{M}$ belongs to the $(2m-1)$-parameter family of graphical ancient trombones, up to rigid motions and time shifts.
\end{corollary}

We recall that, as with total curvature, the entropy is monotone decreasing in time for compact smooth flows. Hence, any blow-up limit at a singularity of a compact flow has finite-entropy. Therefore, ancient finite-entropy flows can be considered as singularity models of flows with finite-entropy, not only in the plane $\mathbb{R}^2$ but also in higher dimensions for the mean curvature flow. In particular, there is extensive literature on the classification of ancient mean curvature flows of hypersurfaces with multiplicity one tangent flow at infinity; see \cite{ADS19,ADS20,BC19,BC21,BL25,CDDHS22,CH24,CHH21,CHH22,CHH23,CHHW22,DH24} from a cylindrical tangent flow and other tangent flows \cite{CCMS24,CCS23}. Here, a \textit{tangent flow at infinity} is a subsequential limit of blow-downs $\mathcal{M}_{\lambda_i}$ as $\lambda_i\to +\infty$, where $\mathcal{M}_\lambda$ is the parabolic rescaling of $\mathcal{M}$ by $\lambda^{-1}$. Similarly, a \textit{tangent flow} at a singularity $X=(x_0,t_0)\in \mathcal{M}$ is a subsequential limit of blow-ups $\mathcal{M}^X_{\lambda_i}$ as $\lambda_i\to 0$, where $\mathcal{M}^X_\lambda := (M-X)_{\lambda}$. Indeed, Bamler-Kleiner \cite{BK23} proved that a closed mean curvature flow in $\mathbb{R}^3$ has a multiplicity one tangent flow. Also, it is well-known that a mean convex flow in $\mathbb{R}^{n+1}$ has a multiplicity one tangent flow as well. Hence, it is natural to study ancient flows with multiplicity one tangent flow at infinity as singularity models for mean curvature flows of hypersurfaces.

However, an immersed flow can develop a singularity whose tangent flow has multiplicity greater than one. For example, under the curve shortening flow in $\mathbb{R}^2$, a figure-eight curve develops a singularity whose tangent flow is a line with multiplicity two. See also Savas-Halilaj and Smoczyk \cite{ShS24LMCF} for examples in Lagrangian mean curvature flow exhibiting singularities with higher multiplicity. On the other hand, Neves \cite{Nev07} showed that the tangent flow at a singularity of the Lagrangian mean curvature flow with single-valued Lagrangian angle is a union of special Lagrangian cones. In particular, in $\mathbb{C}^2$ the tangent flow is a union of planes. It is therefore natural to expect that a singularity model for such Lagrangian mean curvature flows in $\mathbb{C}^2$ should be an ancient solution whose tangent flow at infinity is a union of planes, possibly with multiplicity. Also, the third named author \cite{Su22} constructed Lagrangian translators whose tangent flow at infinity is a union of planes with multiplicity two.
In the absence of higher multiplicity, there are classification and characterization results for ancient Lagrangian mean curvature flow, for instance, by Lambert--Lotay--Schulze \cite{LLS19} and by Lotay--Schulze--Székelyhidi \cite{LSSz24}; however, the higher multiplicity case remains largely unknown. We also note that the tangent flow at infinity of each ancient trombone solution is a line with multiplicity. Therefore, our result establishes, in the lowest dimensional setting, a classification theorem for ancient Lagrangian mean curvature flows whose tangent flow at infinity has higher multiplicity. As such, it provides some evidence toward a better understanding of singularity formation and ancient solutions in Lagrangian mean curvature flow.

 \section{Preliminaries and Notations}

In this section, we introduce notation and some definitions we will frequently use in what follows.

\

Let  $\gamma:\mathbb{L}\times (-\infty, T)\to\mathbb{R}^{2}$ be an ancient solution to the curve shortening flow. We denote the \emph{space-time track} of the flow $\gamma(\cdot, t)$ by $\mathcal{M} = \cup_{t\in(-\infty, T)}(M_{t}, t)$, where $M_{t} = \gamma(\mathbb{L}, t)$ is the \emph{time-slice} at time $t$. 

\ 

We first recall some terminology for some points or subsets in a time-slice $M_{t}$. For more details, see \cite[Section 2]{CSSZ24}. A \emph{vertex} is defined as a critical point of the curvature of $M_t$; it is called \emph{sharp} (resp. \emph{flat}) if it corresponds to a local maximum (resp. local minimum) of the curvature. A curve segment or a curved ray of $M_t$ is called an \emph{edge} if its endpoints are sharp vertices and no sharp vertices in its interior. Similarly, we consider the critical points of the distance function on $M_t$ from a fixed point in $\mathbb{R}^2$; a local maximum (resp. local minimum) is called a \emph{tip} (resp. \emph{knuckle}). A \emph{finger} is a curve segment of $M_t$ whose endpoints are knuckles and whose interior contains no knuckles, while a curved ray of $M_t$ is called a \emph{tail} if it has a unique endpoint that is a knuckle and contains no knuckles in its interior. Instead of the fingers defined above, we will usually work with an \emph{(extended) finger}, defined as the union of two edges sharing a sharp vertex (see Subsection~\ref{sec: glob.geom}).

\ 

Under the finite-entropy assumption
\begin{align}\label{def:entropy}
    {\rm Ent}[\mathcal{M}] := \sup_{t\in(-\infty, T)}\sup_{x_{0}\in\mathbb{R}^{2}, \lambda>0}\int_{M_{t}}\textstyle\frac{1}{\sqrt{4\pi\lambda}}e^{-\frac{|x-x_{0}|^{2}}{4\lambda}}\:d\mathcal{H}^{2}(x)<\infty,
\end{align}
it is shown in \cite{CSSZ24} that the tangent flow at infinity of $\mathcal{M}$ is a unique line through the origin, with integer multiplicity $m+1$. In particular, the rescaled flow $\overline{M}_{\tau} = e^{\tau/2}M_{-e^{-\tau}}$ converges locally smoothly to a line with multiplicity $m+1$ as $\tau\to-\infty$. Moreover, there is an \emph{$\varepsilon$-trombone time} $\tau_{\varepsilon} = -\log(-t_{\varepsilon})\ll -1$ such that whenever $\tau\leq\tau_{\varepsilon}$, the time-slice $\overline{M}_{\tau}$ has almost parallel (depending on $\varepsilon$) edges, and $\overline{M}_{\tau}$ consists of $m$ distinct fingers when $\overline{M}_{\tau}$ is complete; $m+1$ distinct fingers when $\overline{M}_{\tau}$ is closed, each of which has its vertex at the tip of an $\tfrac{\epsilon}{100}$-grim reaper. See \cite[Definition~7.1]{CSSZ24} for the explicit definition.

\ 

We now introduce the notation and conventions used throughout this paper. Let $(x^{1}, x^{2})$ (or sometimes $(y^{1}, y^{2})$ when analyzing the rescale flow) be the coordinates of $\mathbb{R}^{2}$. We will assume that the tangent flow at infinity of an ancient finite-entropy curve shortening flow $\mathcal{M}$ is given by $\{x^{2} = 0\}$ with multiplicity $m+1\geq 3$. For each extended finger $\Gamma_{t}$, there is an associated \emph{finger region} $F(t)$ given by the region enclosed by $\Gamma_{t}$ and the $x^{2}$-axis. By the $C^{\infty}_{loc}$-convergence of $\overline{M}_{\tau}$ to the multiplicity $m+1$ line $\{x^{2} = 0\}$, there is a graphical radius $\rho(\tau)\to\infty$ as $\tau\to-\infty$ such that whenever $\tau\ll -1$, $\overline{M}_{\tau}\cap B_{2\rho(\tau)}$ is a disjoint union of graphical components $\overline{\Sigma}_{\tau}$ over the $x^{1}$-axis, which we call \emph{sheets}. For the unrescaled flow, the corresponding sheets will be denoted by $\Sigma_{t}$. Whenever a curve (segment) is graphical, we call the corresponding graph function the \emph{profile function}.

\section{Sharp asymptotic behaviour of sheets}
\subsection{Setting up}\label{subsec: setup}
Suppose that $A, K_0 >0$, $\vare_0\in (0,\tfrac{1}{100})$, $\tau_0\leq 0$ and a Lipschitz continuous positive function $\rho(\tau)> 10$ is defined for $\tau\leq \tau_0$ such that $\rho(\tau)$ tends to $\infty$ as $\tau$ goes to $-\infty$ and for a.e. $\tau \leq \tau_0$
\begin{align}\tag{R1}\label{hyp: bdd.log.d.rho.1}
    -\tfrac12 \leq \tfrac{\rho'}{\rho}\leq 0.
\end{align}
Suppose that for all $\tau\leq \tau_0$, $u(y, \tau)$ is a profile function of a sheet $\bar\Sigma_\tau$ of rescaled CSF $\bar{M}_\tau$ defined on the interval $(-2\rho(\tau), 2\rho(\tau))$, in particular, $u$ is a solution of the graphical rescaled CSF
\begin{equation}\tag{H1}\label{eq: r.gr.CSF}
    u_{\tau} = \frac{u_{yy}}{1+u_{y}^{2}} - \frac{y}{2}u_{y} + \frac{u}{2} = \cL u - \frac{u_y^2 u_{yy}}{1 + u_y^2}\quad\mbox{on}\quad y\in \big( -2\rho(\tau) , 2\rho(\tau) \big),
\end{equation}
where $\cL := \partial_{y}^2 - \frac{y}{2}\partial_{y} + \frac{1}{2}$ is the linearized operator.
Furthermore, assume $u$ satisfies the bounds:
\begin{align}\tag{H2}\label{hyp: bounds}
    \lVert u(\cdot, \tau)\rVert_{L^\infty((-2\rho(\tau), 2\rho(\tau)))} \leq A e^{\vare_0(\tau-\tau_0)},\quad 
    \|u_y(\cdot, \tau)\|_{L^\infty((-2\rho(\tau), 2\rho(\tau)))} \leq \vare_0,
\end{align}
and the decay condition:
\begin{align}\tag{H3}\label{hyp: decay.cond} 
    \lVert u_{yy}(\cdot, \tau)\rVert_{L^\infty((-2\rho(\tau), 2\rho(\tau)))} \leq K_0/(2\rho(\tau)).
\end{align}
Thinking about the example of translating grim reapers, the lower bound in hypothesis \eqref{hyp: bdd.log.d.rho.1} is reasonable and optimal.
For simplicity of presentation, we may assume that $\tau_0 = 0$ by translating $\cM$ in time and applying an appropriate parabolic rescaling  without changing any of the above hypotheses.

\bigskip

Now, we state the main theorem of this section.
\begin{theorem}[general sharp asymptotic behavior]\label{thm: general asymp line}
Assume $\rho$ satisfies \eqref{hyp: bdd.log.d.rho.1}, and that $u$ satisfies \eqref{hyp: bounds} and \eqref{hyp: decay.cond} in Subsection~\ref{subsec: setup} with constants $A, \vare_0,$ and $K_0$.  
We further assume the following hypotheses on $\rho$: there exist $\tau_1\leq 0$, and \emph{positive} constants $\mu \leq  \tfrac12$ and $B = B(A, K_0, \mu) \gg 1$ such that 
\begin{align}\tag{R2}\label{hyp: ini.cond}
    \rho(0)\geq B,
\end{align}
and
for a.e. $\tau\leq \tau_1$,
\begin{align}\tag{R3}\label{hyp: bdd.log.d.rho.2}
    \tfrac{\rho'(\tau)}{\rho(\tau)} \leq -\mu.
\end{align}
Then, there exists a number $a\in \R$ such that for all $\tau \leq \tau_1$
\begin{align*}
    \|e^{-\tau/2}\hat u - a\|_\cH^2  \leq (1+ a^2) e^{-\rho^2/36}.
\end{align*}
Furthermore, there exists a numerical constant $\beta\in (0,\tfrac{1}{98})$ with the following significance: 
given any $R \geq 1$, there is $\tau_R \leq \tau_1$ $($depending also on $a$) such that for all $\tau\leq \tau_R$,
\begin{equation*}
    \sup_{|y|\leq R}  |e^{-\frac{\tau}{2}}u(y,\tau)-a|  \leq e^{-\beta \rho^2}.
\end{equation*}

\end{theorem}

\bigskip

As a consequence of Theorem~\ref{thm: general asymp line}, we get the following theorem and corollary for ancient flow $\cM$ with finite $\ent[\cM] \geq 3$.
We denote by $u^0(y,\tau)$, $u^1(y,\tau),\cdots, u^m(y,\tau)$ the profile of functions of sheets of $\bar{M}_\tau$ in \cite[Theorem 1.2]{CSSZ24}.

\begin{theorem}\label{thm: sharp asympt line}
    Let $\mathcal{M}$ be an ancient flow whose tangent flow at infinity is the line $\{x^2=0\}$ with $\ent[\cM] = m +1 \geq 3$.
    There are $\{a_0, a_1, \cdots,a_m \}\subset \mathbb{R}$ and a numerical constant $\delta \in (0, \tfrac12)$ such that given $R \geq 1$ there is $\tau_R \ll -1$ $($depending also on $a_i$'s$)$ such that for all $\tau<\tau_R$,
    \begin{equation}
         \sup_{|y|\leq R}  |e^{-\frac{\tau}{2}}u^i(y,\tau)-a_i| \leq \exp(- e^{-2\delta \tau})
    \end{equation}
holds for each $i=0, 1,\cdots,m$.
\end{theorem}

\bigskip

Rescaling back to the original CSF $\mathcal{M}$ via $-t = e^{-\tau}$, we obtain the following estimate for the functions 
\[ 
U^i(x, t) := \sqrt{-t} u^i\big( \tfrac{x}{\sqrt{-t}}, -\log (-t) \big), \quad i = 0, \ldots, m, \]
which represent the profiles of the sheets of $M_t$.
\begin{corollary} \label{cor: sharp asympt line}
Given $R \geq 1$, for $t \leq t_R:=-e^{-\tau_R}$, each sheet $\Sigma^i_t$ of $M_t$ associated with $u^i$ satisfies
\begin{equation}
    \sup_{\lvert x\rvert \leq R\sqrt{-t}}\: \lvert U^i(x, t) - a_i\rvert \leq  \exp(- \lvert t\rvert^{2\delta}).
\end{equation}
\end{corollary}

\bigskip

To process the spectral analysis, we introduce the Gaussian $L^2$-inner product
\begin{equation*}
    \langle f,g\rangle_\mathcal{H}:=\int_{\mathbb{R}}  \frac{fg}{\sqrt{4\pi}} e^{-\frac{y^2}{4}}dy,
\end{equation*}
associated with the Gaussian $L^2$-norm $\|\cdot\|_{\cH}$. 
Since $\mathcal{L}-\frac{1}{2}$ is the Ornstein–Uhlenbeck operator, $\mathcal{L}$ has an eigendecomposition of the Gaussian $L^2$-space; namely, there exists an $\cH$-O.N. basis $\{\varphi_i\}_{i = 1}^{\infty}$ of $\cH$ such that $\cL \varphi_i + \lambda_i \varphi_i = 0$ and $\lambda_i \to \infty$. 
In particular, $\varphi_1(y):=1$ is the only unstable ($\lambda_1 = -\tfrac12$) eigenfunction of $\mathcal{L}$; $\varphi_2(y):=2^{-\frac{1}{2}}y$ is the only neutral ($\lambda_2 = 0$) eigenfunction of $\mathcal{L}$; $\varphi_3(y):=2^{-\frac{3}{2}}(y^2-2)$ is the stable eigenfunction with the least positive eigenvalue $\lambda_3=\frac{1}{2}$. Therefore, for any $f\in \cH$, we can decompose
\begin{align*}
    f = P_+ f + P_0 f + P_- f 
\end{align*}
where $P_+ f =  \langle f , 1\rangle_\cH$, $P_0 f = \tfrac12 \langle f , y \rangle_\cH \, y$, and $P_- f$ are the projections onto the stable, neutral, and stable modes, respectively.
Furthermore, since $\lambda_3 = \tfrac12$, we have
\begin{align}\label{eq: LP-}
    \langle \cL P_- f, P_- f \rangle_\cH \leq -\tfrac12 \|P_- f\|_\cH^2.
\end{align}

\bigskip

\subsection{Spectral analysis}
We begin with several decay estimates for $u$.
The following elementary lemma is useful for the later gradient analysis.
\begin{lemma}
    Let $I = (-\ell, \ell)$ with $\ell>4$ and let $f\in C^2\big((-\ell, \ell)\big)$. 
    For any $\xi\in (0, \ell)$,
    \begin{align}\label{eq: interp.ineq.1}
        \lVert f'\rVert_{L^\infty(I)} 
        \leq \tfrac{2}{\xi} \lVert f\rVert_{L^\infty(I)} + \tfrac{\xi}{2}\lVert f''\rVert_{L^\infty(I)}.
    \end{align}
    Furthermore, if $\ell\,\lVert f''\rVert_{L^\infty(I)} \leq K_0$ for some $K_0 < \infty$, then
    \begin{align}\label{eq: interp.ineq.2}
        \lVert f'\rVert_{L^\infty(I)} 
        \leq \tfrac{1}{\sqrt{\ell}}\big(\lVert f\rVert_{L^\infty(I)} + K_0\big).
    \end{align}
\end{lemma}

\begin{proof}
    Let $x\in (-\ell, \ell)$ and let $\xi\in (0, \ell)$. 
    At least one of $x\pm \xi$ lies in $(-\ell, \ell)$ and let $y$ be such $x \pm \xi$.
    By Taylor's theorem, 
    \begin{equation*}
        f(y) = f(x) + (y-x)f'(x) + \int_{x}^{y}(y - s)f''(s)\, ds,
    \end{equation*}
    and thus
    \begin{align*}
        \lvert f'(x)\rvert = \left\lvert \tfrac{f(y)-f(x)}{y-x} - \tfrac{1}{y-x}\int_{x}^y (y-s) f''(s) \, ds \right\rvert
        \leq \tfrac{2}{\xi}\lVert f\rVert_{L^\infty((0, \ell))} + \tfrac{\xi}{2} \lVert f''\rVert_{L^\infty((0, \ell))}.
    \end{align*}
    Furthermore, \eqref{eq: interp.ineq.2} can be derived from \eqref{eq: interp.ineq.1} with $\xi = 2\sqrt{\ell} < \ell$ provided that the extra condition holds.
\end{proof}

\bigskip
Applying inequality \eqref{eq: interp.ineq.2}, the hypotheses \eqref{hyp: bounds} and \eqref{hyp: decay.cond} yield a gradient decay estimate: for all $\tau \leq \tau_0$,
\begin{align}\label{eq: du.decay}
    \lVert u_y(\cdot, \tau)\rVert_{L^\infty((-2\rho(\tau), 2\rho(\tau)))} 
    \leq \tfrac{1}{\sqrt{\rho}}\big( A + K_0\big).
\end{align}
Therefore, when $\rho \gg (A + K_0)^2$, the small gradient condition stated in \eqref{hyp: bounds} becomes redundant.

\ 

\begin{proposition}\label{prop: d3.u decay}
    Suppose that $\rho$ and $u$ satisfy \eqref{hyp: bdd.log.d.rho.1}, 
    \eqref{eq: r.gr.CSF}
    ,\eqref{hyp: bounds}, and \eqref{hyp: decay.cond}. There exists a universal constant $K_1$ depending only on $K_0$ such that for all $\tau \leq \tau_0$
    \begin{align}\label{eq: d3u.decay}
        \lVert u_{yyy}(\cdot, \tau)\rVert_{L^\infty((-\rho, \rho))} \leq K_1/\rho^2.
    \end{align}
\end{proposition}

\begin{proof}
    Let $\cG = \cup_{t}\,\Sigma_t\times\{t\}$ denote the sub-flow of $\cM$ represented by $u$.
    Note that the curvature of $\bar\Sigma_\tau$ is given by $\bar\kappa = u_{yy}(1 + u_y^2)^{-\frac32}$.
    Then \eqref{hyp: bounds} and \eqref{hyp: decay.cond} imply that for $\tau\leq \tau_0$,
    \begin{align*}
        \sup\: \big\{ \lvert \bar\kappa(x) \rvert \::\: x\in B\big(0, 2\rho(\tau)\big) \cap \bar{\Sigma}_\tau \big\}\leq K_0 /(2\rho(\tau)),
    \end{align*}
    which is equivalent to that for any $t \leq -e^{-\tau_0}$
    \begin{align*}
        \sup\: \big\{ \lvert \kappa(x) \rvert \::\: x\in B\big(0, 2\sqrt{-t}\rho(t)\big) \cap \bar{\Sigma}_\tau \big\}\leq K_0 /(2\sqrt{-t}\rho(t)),
    \end{align*}
    where by abuse of notation $\rho(t)$ means $\rho\circ \tau(t)$.
    Since $2\sqrt{-t}\rho(t)$ is monotone decreasing in $t$, for any $t' \leq -e^{-\tau_0}$ we have the curvature estimate in the parabolic ball: for any $X\in P\big((0,t'), 2\sqrt{-t'}\rho(t')\big) \cap \cG$,
    \begin{align*}
        \lvert \kappa(X) \rvert \leq K_0 /(2\sqrt{-t'}\rho(t')).
    \end{align*}
    Then by Shi-type interior estimate \cite[Proposition 3.22]{Eck2004RTM}, there exists a numerical constant $c_1 = c_1(K_0)$ such that for any $t'\leq -e^{-\tau_0}$, for any $X\in P\big((0,t'), \sqrt{-t'}\rho(t')\big) \cap \cG$,
    \begin{align*}
        \big\lvert \kappa_s(X) \big\rvert \leq c_1/ (\sqrt{-t'}\rho(t'))^2.
    \end{align*}
    Rescaling it back to $\bar \Sigma_\tau$ gives that for any $\tau\leq \tau_0$, 
    \begin{align}\label{eq: d.kappa 2}
        \sup\: \big\{ \lvert \bar\kappa_s(x) \rvert \::\: x\in B\big(0, \rho(\tau)\big) \cap \bar{\Sigma}_\tau \big\}\leq c_1 /\rho(\tau)^2.
    \end{align}
    Note that 
    \begin{align*}
        \bar\kappa_s = u_{yyy}(1 + u_y^2)^{-2} - 3u_y(u_{yy})^2 (1 + u_y^2)^{-3},
    \end{align*}
    where by \eqref{hyp: bounds} and \eqref{hyp: decay.cond} the absolute value of the second term in the expansion of $\bar\kappa_s$ is bounded above by $3\vare_0K_0^2\rho^{-2}$. 
    Therefore, there exists a numerical constant $K_1 = K(K_0)$ such that $\lvert w_{yyy}\rvert \leq K_1/\rho^2$ in $(-\rho, \rho)$ for all $\tau \leq \tau_0$.
\end{proof}

\bigskip
We will extend $u(\cdot, \tau)$ to a function on $\R$ by applying a cut-off function.
Let $\eta:\mathbb{R} \to [0, 1]$ be a fixed cut-off function such that $\eta(s)\equiv 1$ for $|s|\leq 1/2$, and $\eta(s)\equiv 0$ for $|s|\geq 3/4$ and such that $\lvert \eta' \rvert \leq 8, \lvert \eta'' \rvert \leq 32$ everywhere.
Then, for all $y\in \R$ and $\tau \leq 0$, we define
\begin{equation*}
    \hat u(y, \tau) := u(y, \tau)\, \eta\big(y/\rho(\tau)\big),
\end{equation*}
and define the error term
\begin{equation*}
    E := \hat u_\tau - \cL \hat u.
\end{equation*}
A direct computation gives a decomposition $E = E_1 + E_2$ where
\begin{equation*}
    E_1 := - \frac{\hat u_y^{2}{u}_{yy}}{1+u_{y}^{2}}
\end{equation*}
and
\begin{equation*}
    E_2 := \tfrac{ u_{yy}}{1 + u_y^2}\Big( \eta(\eta - 1)u_y^2 + 2u u_y \eta \tfrac{\eta'}{\rho} + u^2 \tfrac{\lvert \eta'\rvert^2}{\rho^2} \Big) + \tfrac{\eta'}{\rho}\Big( (\tfrac12 - \tfrac{\rho'}{\rho}) uy - 2u_y \Big) - u \tfrac{\eta''}{\rho^2}.
\end{equation*}
Here and thereafter, for brevity, we omit the composite with $y/\rho$ in $\eta, \eta', \eta''$. 
We further factor out $E_1 = \hat u_y\, \tilde{E}$ where
\begin{align*}
     \tilde{E} := -\frac{\hat{u}_{y}u_{yy}}{1+u_{y}^{2}}.
\end{align*}
A direct computation gives
\begin{align*}
    |\tilde{E}_{y}| = \left\lvert\frac{(\hat{u}_{yy}u_{yy}+\hat{u}_{y}u_{yyy})(1+u_{y}^{2}) - 2\hat{u}_{y}u_{y}u_{yy}^{2}}{(1+u_{y}^{2})^{2}}\right\rvert,
\end{align*}
which will be needed after integration by parts.

\bigskip
The error $E$, in an integral sense, can be bounded by a quadratic term of $\|\hat{u}\|_{\cH}$ up to an exponentially decaying term.
\begin{lemma}\label{lem: E error term}
    Fix $\tau$. Suppose that $\rho > 10$. We have
    \begin{align}
        \lvert\langle E_{1}, P_{*}\hat{u}\rangle_{\cH} \rvert &\leq \rho\, \|\tilde{E}\|_{C^1(\R)}\,\|\hat u\|_\cH^2, \label{eq: proj.E1}\\
        \lvert\langle E_{2}, P_{*}\hat{u}\rangle_{\cH} \rvert &\leq \tfrac12 \rho^{-\frac12}\|E_2\|_{C^0(\R)}^2\,\|\hat{u}\|_{\cH}^2 + \tfrac12 \rho^{\frac32}\, e^{-\rho^2/16}. \label{eq: proj.E2} 
    \end{align}
    where $* = +, 0, -$.
\end{lemma}
\begin{proof}
    Recalling $P_+ \hat u =  \langle \hat u , 1\rangle_\cH$ and $P_0 \hat u = \tfrac12 \langle \hat u , y \rangle_\cH \, y$, we have $(P_{+}\hat{u})_{y} = 0$ and  $(P_{0}\hat{u})_{y} = \tfrac12\langle\hat{u}, y\rangle_{\cH}$. Thus, for $* = +, 0$,
    \begin{align}\label{eq: nonpositive mode inverse poincare}
        \|(P_{*}\hat{u})_{y}\|^{2}_{\cH}\leq \tfrac12 \|P_{*}\hat{u}\|^{2}_{\cH},
    \end{align}
    and then integration by parts gives
    \begin{align*}
        |\langle E_{1}, P_{*}\hat{u}\rangle_{\cH}| &= \left|\int_{\mathbb{R}}\hat{u}_{y}\, \tilde{E} \:(P_{*}\hat{u})\:e^{-\frac{y^{2}}{4}}dy\right|\\
        &=\left|\int_{\R} \hat{u}\left(\tilde{E}_{y} \,(P_{*}\hat{u}) + \tilde{E}\,(P_{*}\hat{u})_{y} - \frac{y}{2}\tilde{E}\, (P_{*}\hat{u})\right)\,e^{-\frac{y^{2}}{4}}dy\right|.
    \end{align*}
    Using (\ref{eq: nonpositive mode inverse poincare}) and H\"older inequality, we find that for $* = +, 0$,
    \begin{align*}
        |\langle E_{1}, P_{*}\hat{u}\rangle_{\cH}|\leq \rho\,\|\tilde{E}\|_{C^1(\R)}\,\|\hat{u}\|_{\cH}\,\|P_{*}\hat{u}\|_{\cH}.
    \end{align*}
    We remark that the term with $y\tilde{E}$ supported on the interval $(-\rho, \rho)$ has the worst estimate.

    Next, we show the inequality \eqref{eq: proj.E1} holds for $* = -$. Observe that by integration by parts,
    \begin{align*}
        \langle E_{1}, P_{-}\hat{u}\rangle_{\cH} &= \int_{\mathbb{R}}\hat{u}_{y}\tilde{E}\: (P_{-}\hat{u})\:e^{-\frac{y^{2}}{4}}dy\\
        &=\int_{\mathbb{R}}\left[(P_{+}\hat{u})_{y} + (P_{0}\hat{u})_{y} + (P_{-}\hat{u})_{y}\right] (P_{-}\hat{u}) \:\tilde{E}\:e^{-\frac{y^{2}}{4}}dy\\
        &=\tfrac12 \langle\hat{u}, y\rangle_{\cH}\int_{\mathbb{R}} (P_{-}\hat{u})\:\tilde{E}\:e^{-\frac{y^{2}}{4}}dy + \int_{\mathbb{R}} \tfrac{1}{2}[(P_{-}\hat{u})^{2}]_{y}\tilde{E}\:e^{-\frac{y^{2}}{4}}dy\\
        &= \tfrac12 \langle\hat{u}, y\rangle_{\cH}\int_{\mathbb{R}} (P_{-}\hat{u})\:\tilde{E}\:e^{-\frac{y^{2}}{4}}dy  - \tfrac{1}{2}\int_{\mathbb{R}}(P_{-}\hat{u})^{2}\left(\tilde{E}_{y} - \tfrac{y}{2}\tilde{E}\right)\:e^{-\frac{y^{2}}{4}}dy.
    \end{align*}
    Therefore, since $\rho > 10$,
    \begin{align*}
        |\langle E_{1}, P_{-}\hat{u}\rangle_{\cH}|\leq  \|\tilde{E}\|_{C^1(\R)}\,\|P_{0}\hat{u}\|_{\cH}\|P_{-}\hat{u}\|_{\cH} + \tfrac12\rho\|\tilde{E}\|_{C^1(\R)}\|P_{-}\hat{u}\|^{2}_{\cH}
        \leq \rho \|\tilde{E}\|_{C^1(\R)}\, \|\hat{u}\|_{\cH}^2.
    \end{align*}

    Finally, we show \eqref{eq: proj.E2}. 
    Since the support of $E_{2}$ is contained in $\{\tfrac{\rho}{2}\leq |y|\leq \tfrac{3\rho}{4}\}$, we may use the decay of Gaussian kernel to obtain the estimate
    \begin{align*}
        \left|\langle E_{2}, P_{*}\hat{u}\rangle_{\mathcal{H}}\right| &\leq \int_{\frac{\rho}{2}\leq |y|\leq \frac{3\rho}{4}} |E_2||P_{*}\hat{u}|\:e^{-\frac{y^{2}}{4}}\:dy\\
        &\leq \|E_2\|_{C^0(\R)}\Big(\int_{\frac{\rho}{2}\leq |y|\leq \frac{3\rho}{4}}e^{-\frac{y^{2}}{4}}\:dy\Big)^{\frac12} \|\hat{u}\|_{\cH}\\
        &\leq \rho^{\frac12} e^{-\rho^2/32}\, \|E_2\|_{C^0(\R)}\, \|\hat u\|_\cH\\
        &\leq \tfrac12 \rho^{-\frac12}\|E_2\|_{C^0(\R)}^2\,\|\hat{u}\|_{\cH}^2 + \tfrac12 \rho^{\frac32}\, e^{-\rho^2/16}.
    \end{align*}
\end{proof}

\begin{proposition}\label{prop: E error term}
    There exists a universal constant $K = K(A, K_0)>1$ such that for all $\tau \leq 0$,
    \begin{align}\label{eq: coeff decay}
    2\lvert\langle E, P_{*}\hat{u}\rangle_{\cH} \rvert 
    &\leq  \tfrac{K}{\sqrt{\rho}} \big(\|\hat{u}\|_{\cH}^2 + \rho^2 e^{-\rho^2/16}\big) .
\end{align}
\end{proposition}

\begin{proof}
Note that the coefficients of $\lVert \hat u\rVert^2_\cH$ in Lemma~\ref{lem: E error term} decay to zero as $\tau$ goes to $-\infty$.
Owing to the $C^2$ decay conditions \eqref{hyp: decay.cond}, \eqref{eq: du.decay}, and \eqref{eq: d3u.decay} of $u_y$ and hypothesis \eqref{hyp: bdd.log.d.rho.1} on $\rho$, there exists a universal constant $K = K(A, K_0)>1$ such that for all $\tau \leq 0$,
\begin{align}
    \|\tilde E\|_{C^1(\R)} \leq \tfrac{K}{4} \rho^{-\frac32},
\end{align}
and 
\begin{align}
    \|E_2\|_{C^0(\R)}^2 \leq \tfrac{K}{2},
\end{align}
where the term with $y$ in $E_2$ has the worst estimate.
Therefore, Lemma~\ref{lem: E error term} implies the claim.

\end{proof}

\ 

\begin{proof}[Proof of Theorem~\ref{thm: general asymp line}]
It follows from \eqref{eq: r.gr.CSF} that each mode of $u$, respectively, satisfies
\begin{align}
    \tfrac{d}{d\tau}\|P_{+}\hat{u}\|^{2}_{\mathcal{H}} &= \|P_{+}\hat{u}\|^{2}_{\mathcal{H}} + 2\langle E, P_{+}\hat{u}\rangle_{\mathcal{H}} ,\label{eq: u.+mode}\\
    \tfrac{d}{d\tau}\|P_{0}\hat{u}\|^{2}_{\mathcal{H}} &= 2\langle E, P_{0}\hat{u}\rangle_{\mathcal{H}} ,\label{eq: u.0mode}\\
    \tfrac{d}{d\tau}\|P_{-}\hat{u}\|^{2}_{\mathcal{H}} &\leq -\|P_{-}\hat{u}\|^{2}_{\mathcal{H}} + 2\langle E, P_{-}\hat{u}\rangle_{\mathcal{H}}.\label{eq: u.-mode}
\end{align}
Set 
\begin{align}
    W_+(\tau) := \|P_+ \hat u\|^2_\cH + \rho^2 e^{-\rho^2/16} ,\quad  W_0(\tau) := \|P_0 \hat u\|^2_\cH, \quad  W_-(\tau) := \|P_- \hat u\|^2_\cH.
\end{align}
Remark that we add $\rho^2 e^{-\rho^2/16}$ in $W_+$ to absorb the error terms generated by cutting off $u$ in the system of differential inequalities.

By using the hypotheses \eqref{hyp: ini.cond} and \eqref{hyp: bdd.log.d.rho.2} on $\rho$, for all $\tau \leq \tau_1$,
\begin{align}\label{hyp: rho.exp.growth}
    \rho(\tau) \geq B e^{-\mu \tau}.
\end{align}
Then, Proposition~\ref{prop: E error term} gives the integral estimate of error terms:
\begin{align}\label{eq: proj.E}
    2\lvert\langle E, P_{*}\hat{u}\rangle_{\cH} \rvert 
    &\leq  \tfrac{K}{\sqrt{B}}e^{\mu \tau/2}\big(\|\hat{u}\|_{\cH}^2 + \rho^2 e^{-\rho^2/16}\big).
\end{align}
Hence, a computation using \eqref{eq: u.+mode} and hypotheses \eqref{hyp: bdd.log.d.rho.1}, \eqref{hyp: bdd.log.d.rho.2} on $\rho$ give
\begin{align}
    \tfrac{d}{d\tau}W_+ - W_+ &= 2\langle E, P_{+}\hat{u}\rangle_{\mathcal{H}}  
    + [-\tfrac{\rho'}{\rho}(\tfrac{\rho^2}{8}-2) - 1]\rho^2 e^{-\rho^2/16}\label{eq: d.W+}\\
    &\geq -\tfrac{K}{\sqrt{B}}e^{\mu \tau/2}\big(\|\hat{u}\|_{\cH}^2 + \rho^2 e^{-\rho^2/16}\big) 
    + [\tfrac{\mu B^2}{8} - 2]\rho^2 e^{-\rho^2/16}\notag\\
    &\geq - \tfrac{K}{\sqrt{B}}e^{\mu \tau/2}\big(\|\hat{u}\|_{\cH}^2 + \rho^2 e^{-\rho^2/16}\big)\notag
\end{align}
provided that we take $B>\sqrt{16/\mu}$.
Then, by applying \eqref{eq: proj.E} to \eqref{eq: u.0mode} and \eqref{eq: u.-mode}, we have the following system of differential inequalities: for all $\tau \leq \tau_1$
\begin{align}
    \tfrac{d}{d\tau}W_+ - W_+ &\geq  -\tfrac{K}{\sqrt{B}}e^{\mu \tau/2}(W_+ + W_0 + W_-),\label{eq: MZ+}\\
    \left|\tfrac{d}{d\tau} W_0 \right| &\leq \tfrac{K}{\sqrt{B}}e^{\mu \tau/2}(W_+ + W_0 + W_-),\label{eq: MZ0}\\
    \tfrac{d}{d\tau} W_- + W_- &\leq \tfrac{K}{\sqrt{B}}e^{\mu \tau/2}(W_+ + W_0 + W_-).\label{eq: MZ-}
\end{align}

Since $W_+ + W_0 + W_- >0$ and $W_+, W_0, W_-\to 0$ as $\tau\to-\infty$ due to hypothesis \eqref{hyp: bounds}, by the ODE lemma (\cite[Lemma A.1]{MZ98OEB}), there exists a numerical constant $C$ such that if we choose $B \geq C^2 K^2$, then
\begin{align}\label{eq: MZ1}
    W_++W_- \leq e^{\mu \tau/2} W_0 
\end{align}
or
\begin{align}\label{eq: MZ2}
    W_0+W_-\leq e^{\mu \tau/2} W_+.
\end{align}

We claim that (\ref{eq: MZ1}) cannot occur.
Suppose the contrary. We deduce from (\ref{eq: MZ0}) that
\begin{align*}
    \left|\tfrac{d}{d\tau}W_0 \right| \leq e^{\mu \tau/2}W_0.
\end{align*}
Hence, $|(\log W_0)'|\leq e^{\mu \tau/2}$ and therefore  $|\log W_0|\leq C_1$ with some constant $C_1>0$ for all $\tau \leq \tau_1$. Then, $W_0$ can not converge to zero, a contradiction.

Now, we only need to consider (\ref{eq: MZ2}). 
We assume that $B$ in \eqref{hyp: rho.exp.growth} is sufficiently large such that for all $\tau\leq \tau_1$
\begin{align}
    \rho^2 e^{-\rho^2/16} \leq e^{-\rho^2/25}.
\end{align}
Combining this with \eqref{eq: u.+mode} and \eqref{eq: proj.E} gives that $\tau \leq \tau_1$
\begin{align*}
    \tfrac{d}{d\tau} \|P_+ \hat u\|_\cH^2 \leq \|P_+ \hat u\|_\cH^2 + e^{\mu \tau/2}\|P_+ \hat u\|_\cH^2 + e^{-\rho^2/25}.
\end{align*}
Applying the method of integrating factor gives 
\begin{align*}
    \tfrac{d}{d\tau} \big[\exp\big(-\tau - \tfrac{2}{\mu}e^{\mu \tau/2}\big) \|P_+ \hat u\|_\cH^2 \big] \leq \exp\big(-\tau - \tfrac{2}{\mu}e^{\mu \tau/2} - \tfrac{\rho^2}{25}\big) \leq \exp\big(-\tfrac{2}{\mu}e^{\mu\tau_1/2}\big)e^{-\rho^2/36}.
\end{align*}
Here, by taking $B$ large in \eqref{hyp: rho.exp.growth}, we can assume that $\rho^2/25-\rho^2/36 \geq -2\tau$ for all $\tau \leq \tau_1$.
It follows from the comparison theorem for improper integrals that the limit
\begin{align*}
    \lim_{\tau \to -\infty} \exp\big(-\tau - \tfrac{2}{\mu}e^{\mu \tau/2}\big) \|P_+ \hat u\|_\cH^2 = \lim_{\tau\to -\infty} e^{-\tau}\|P_+ \hat u\|_\cH^2
\end{align*}
exists. 
Since $P_+\hat u = \langle \hat u, 1\rangle_\cH$ is continuous in $\tau$, there exists $a\in \R$ such that $e^{-\tau/2} P_+\hat u$ converges to $a$ as $\tau$ tends to $-\infty$.
Therefore, \eqref{eq: MZ2} implies that $\|e^{-\tau/2} \hat{u} - a\|_\cH^2 \to 0$ as $\tau$ tends to $-\infty$.

Next, we want to estimate the rate of convergence. 
It is straightforward to verify that $w := u - ae^{\tau/2}$ is also a solution to rescaled graphical CSF \eqref{eq: r.gr.CSF}. 
Furthermore, $e^{-\tau/2}P_+ \hat{w}$ converges to 0 as $\tau \to -\infty$ since  
\begin{equation*}
\|\hat{a} - a\|_{\mathcal{H}}^2 = \frac{1}{\sqrt{4\pi}} \int_{\lvert y \rvert \geq \rho/2} a^2 \Big( 1 - \eta \big( \tfrac{y}{\rho} \big) \Big)^2 e^{-\frac{y^2}{4}} \, dy \leq \frac{4a^2}{\rho \sqrt{\pi}} e^{-\rho^2/16},
\end{equation*}
where $\hat a = a\eta(y/\rho)$ is defined as before.
Applying the same spectral dynamic analysis to $\hat{w}$ in previous paragraphs, we end up with
\begin{align*}
    \tfrac{d}{d\tau} \big[\exp\big(-\tau - \tfrac{2}{\mu}e^{\mu \tau/2}\big) \|P_+ \hat w\|_\cH^2 \big] \leq \exp\big(-\tfrac{2}{\mu}e^{\mu\tau_1/2}\big) e^{-\rho^2/36}.
\end{align*}
Integrating over the interval $[\tau', \tau]$ yields
\begin{align*}
    &\exp\left(-\tau - \tfrac{2}{\mu}e^{\mu \tau/2}\right)\|P_{+}\hat{w}(\cdot, \tau)\|^{2}_{\mathcal{H}} - \exp\left(-\tau' 
    - \tfrac{2}{\mu}e^{\mu \tau'/2}\right)\|P_{+}\hat{w}(\cdot, \tau')\|^{2}_{\mathcal{H}}\\
    \leq &\exp\big(-\tfrac{2}{\mu}e^{\mu\tau_1/2}\big)\int_{\tau'}^{\tau} e^{- \rho(s)^2/36}\:ds.
\end{align*}
Using rebalancing condition that $\lim_{\tau'\to-\infty}e^{-\tau'}\|P_{+}\hat{w}(\cdot, \tau')\|^{2}_{\mathcal{H}} = 0$, for all $\tau \leq \tau_1$
\begin{align}\label{eq: e.w+}
    e^{-\tau}\|P_{+}\hat{w}(\cdot, \tau)\|^{2}_{\mathcal{H}}\leq \int_{-\infty}^{\tau}e^{-\rho^{2}(s)/36}\:ds.
\end{align}

We need to analyze $\int_{-\infty}^{\tau}e^{-\rho^{2}(s)/36}\:ds$. 
Set the substitution $\xi = \rho(s)^2$ with $d\xi = 2\rho\rho'ds$. 
Then the hypothesis \eqref{hyp: bdd.log.d.rho.2} implies that $2\rho\rho' \leq -2\mu\xi$
and 
\begin{align*}
    \int_{-\infty}^{\tau}e^{-\rho^{2}(s)/36}\:ds 
    \leq (2\mu)^{-1}\int_{\rho(\tau)^2}^\infty \xi^{-1} e^{-\xi/36}\:d\xi.
\end{align*}
By taking $B$ large such that $(2\mu B^2)^{-1}  \leq \tfrac{1}{144}$, we obtain
\begin{align*}
    \int_{-\infty}^{\tau}e^{-\rho^{2}(s)/36}\:ds 
    \leq \int_{\rho(\tau)^2}^\infty \tfrac{1}{144}e^{-\xi/36}\:d\xi = \tfrac{1}{4}e^{-\rho(\tau)^2/36}.
\end{align*}
Putting it back to \eqref{eq: e.w+} yields that for all $\tau\leq \tau_1$
\begin{align*}
    e^{-\tau}\|P_{+}\hat{w}(\cdot, \tau)\|^{2}_{\mathcal{H}}\leq \tfrac{1}{4}e^{-\rho^2/36}.
\end{align*}
From the dominance of unstable mode (\ref{eq: MZ2}), for $\tau\leq \tau_1$
\begin{align*}
    \|e^{-\tau/2} \hat w(\cdot, \tau)\|^2_\cH \leq \tfrac{1}{3}e^{-\rho(\tau)^2/36}.
\end{align*}
Therefore, by taking $B\gg 1$ (independent of $a$),
\begin{align*}
    \|e^{-\tau/2}\hat u - a\|_\cH^2 \leq 2\|e^{-\tau/2}\hat w\|_\cH^2 + 2\|\hat a - a\|_\cH^2 \leq (1+ a^2) e^{-\rho^2/36}.
\end{align*}

Lastly, we want to show $L^{\infty}$-convergence.
Let $R> 0$. 
For $\tau \leq \tau_1$ for which $\rho(\tau) > 4R$,
\begin{align*}
    \| e^{-\tau/2} u - a \|^2_{L^2([-2R, 2R])}\leq e^{R^2} (1+ \lvert a\rvert) e^{-\rho^2/49}.
\end{align*}
Using the assumption \eqref{hyp: bounds} on small gradient that also holds in the region $[-2R, 2R]$ and the standard interpolation inequality \cite[Lemma B.1]{CM15}, there exist $\tau_R\ll -1$ and a constant $\beta\in (0, \tfrac{1}{98})$ such that $\| e^{-\tau/2} u(\cdot,\tau) - a \|_{L^\infty([-R, R])} \leq e^{-\beta\rho(\tau)^2}$ for all $\tau \leq \tau_R$. 
\end{proof}

\bigskip

In the end of this section, we prove Theorem~\ref{thm: sharp asympt line} by using Theorem~\ref{thm: general asymp line} and verifying its hypotheses. 

\begin{proof}[Proof of Theorem~\ref{thm: sharp asympt line}]
We may choose the graphical radius $\rho(\tau) = e^{-\delta \tau}/2$ with $\delta \in (0,\tfrac12)$ in \cite[Theorem 1.2]{CSSZ24}, which clearly satisfies \eqref{hyp: bdd.log.d.rho.1}, \eqref{hyp: bdd.log.d.rho.2} with $\mu  = \delta$, and \eqref{hyp: ini.cond} provided that $\tau_1\ll -1$.
Combining \cite[Theorem 1.2]{CSSZ24} and \cite[Theorem 7.3]{CSSZ24} leads to hypotheses \eqref{hyp: bounds} and \eqref{hyp: decay.cond} on $(-2\rho, 2\rho)$ (by choosing a slightly smaller $\delta$).
Therefore, the result follows from Theorem~\ref{thm: general asymp line}.
\end{proof}

\section{Coarse asymptotics in space}
To obtain exponential convergence of the edges to their asymptotes as a grim reaper or a trombone, it is critical to extend the graphical radius up to an optimal rate. 
In particular, we desire a graphical radius $\varrho(t) = \sqrt{-t}\rho(t) = O(\lvert t \rvert)$ with the optimal rate $\mu = \tfrac12$ in \eqref{hyp: bdd.log.d.rho.2}, within which the profile function $U(x,t)$ of $M_t$ and its derivative $U_x$ have uniform $L^\infty$-estimates, implying \eqref{hyp: bounds}, and furthermore, $U_{xx}$ or $\kappa$ satisfies a certain estimate implying \eqref{hyp: decay.cond}.
This section is devoted to obtaining the above essential estimates to obtain Theorem~\ref{thm: exp.conv.away.tip}.

\subsection{Solution in strip}

\begin{lemma}[asymptotic slope]\label{lem: asymptotic slope} For every $\delta\in(0,1)$, there exists $\Lambda=\Lambda(\delta)<\infty$ and $t_0\in \mathbb{R}$ such that 
\begin{align} \label{asymptotic slope}
    \sup\left\{\left|\frac{x^2}{x^1}\right|: (x^1, x^2)\in M_t\right\}<\delta
\end{align}
   provided that $t\le t_0$ and $|x^1|\ge \Lambda\sqrt{-t}$. 
\end{lemma}
\begin{proof}
    We recall $a_1, \dots ,a_m$ in Corollary \ref{cor: sharp asympt line}. Take a constant $A$ such that 
    \begin{align}
        A=1+\max\{|a_1|,\dots, |a_m|\}.
    \end{align}
    In addition, take constants $t_0< 0$ and $\Lambda_0=\Lambda_0(\delta)<\infty$ such that
    \begin{align}
        \frac{A}{\Lambda_0\sqrt{-t_0}} +\frac{2}{\Lambda_0^2}<\delta.
    \end{align}
    Then, by Corollary \ref{cor: sharp asympt line}, for any $r\ge \Lambda_0$, there is $t_r\le t_0$ such that $(x^1, x^2)\in M_t$ with $|x^1|\leq  2r\sqrt{-t}$  satisfies
    \begin{align} \label{x^2 asymptotic bound}
        -A\le x^2 \le A.
    \end{align} 
    for all $t\leq t_{r}$. \\

    Let $S(r):=r\sqrt{-t_r}$. We claim that (\ref{asymptotic slope}) satisfies for all $t\le t_0$ and $|x^1|\ge S(\Lambda_0)$. Fix $x^1 = S(r)$. By (\ref{x^2 asymptotic bound}), (\ref{asymptotic slope}) holds for $t\le T_r$. Consider the shrinking circles
    \begin{align}
        N_t^1:=\left\{(x^1, x^2): (x^1-S(r))^2+(x^2-S(r)-A)^2=S(r)^2-2(t-t_r)\right\}
    \end{align}
    and 
    \begin{align}
         N_t^2:=\left\{(x^1, x^2): (x^1-S(r))^2+(x^2+S(r)+A)^2=S(r)^2-2(t-t_r)\right\}
    \end{align}
    for $t\in [t_r,0]$. Then, these two circles cannot intersect with $M_t$. For $(S(r), x^2)\in M_t$, we have
    \begin{align}
        |x^2| &\le A+ S(r)-\sqrt{S(r)^2-2(t-T_r)}\\
        &= A+\frac{2(t-T_r)}{S(r)+\sqrt{S(r)^2-2(t-T_r)}}\le A+\frac{2}{r}\sqrt{-T_r}.
    \end{align}
    Thus, for $t\in[t_r, t_0]$, we have 
    \begin{align}
        |x^2|\le \left(\frac{A}{\Lambda_0\sqrt{-t_0}}+\frac{2}{\Lambda_0^2}\right)|x^1| < \delta |x^1|.
    \end{align}
    Thus, our claim holds for $x^1\ge S(\Lambda_0)$.
    By a similar argument, our claim works for $x^1\le -S(\Lambda_0)$. By taking $\Lambda=S(\Lambda_0)/\sqrt{-t_0}$, we obtained the desired conclusion.
\end{proof}

\begin{proposition}[tail in a strip]\label{prop: tail in strip}
    Let $\Phi_t$ be a right-going tail of $M_t$ and let $p^1\geq 0$ be a fixed value. Denote by $p(t):=(p^1,p^2(t))$ the intersection of $\Phi_t$ and $x^1 = p^1$. Then,
    \begin{equation*}
        \Phi_t \cap \big\{(x^1,x^2)\::\: x^1 \geq p^1 \big\} \subset \big\{(x^1,x^2)\::\:  \inf_{s\leq t}\, p^2(s) \leq x^2 \leq \sup_{s\leq t}\, p^2(s) \big\}.
    \end{equation*}
    for all $t\le t_0$, where $t_0$ is a constant in Lemma \ref{lem: asymptotic slope}.
\end{proposition}
\begin{proof}
 
    By Corollary \ref{cor: sharp asympt line}, $\Phi_t$ contains a part of a sheet, say $\Sigma_t^i$, and
    \begin{align}
        \inf_{s\le t} \: p^2(s)\le a_i \le \sup_{s\le t} \: p^2(s).
    \end{align}

    Let $\delta, \varepsilon>0$. We consider a graph of the line with a slope $\delta$ passing through $(p^1, \sup_{s\le t}\: p^2(s)+\varepsilon)$ and a graph of the line with a slope $-\delta$ passing through $(p^1, \inf_{s\le t}\: p^2(s)-\varepsilon)$. By Lemma \ref{lem: asymptotic slope}, there is $\Lambda=\Lambda(\delta)>0$ such that for all $s\le t_0$,
    \begin{align} \label{distant part}
    \Gamma_s\cap\left\{x^1\ge \Lambda\sqrt{-s}\right\}    
    \end{align}
     is bounded by the two graphs and $\{x^1=\Lambda\sqrt{-s}\}$. Then, for all $s_0\le s\le t$, both of
     \begin{align*}
         \Phi_s\cap \{x^1=p^1\}
     \end{align*}
     and 
     \begin{align*}
         \Phi_s\cap \{x^1=\Lambda\sqrt{-s_0}\}
     \end{align*}
     are bounded by the two graphs. On the other hand, by Corollary \ref{cor: sharp asympt line}, for sufficiently large $-s_0$ such that
     \begin{align*}
         \Phi_{s_0} \cap \{p^1\le x^1\le \Lambda\sqrt{-s_0}\}
     \end{align*}
    is bounded by the two graphs, $\{x^1=p^1\}$, and $\{x^1=\Lambda\sqrt{-s_0}\}$.
    Therefore, by maximum principle and (\ref{distant part}), $\Phi_{t}$ is bounded by the two graphs and $\{x^1=p^1\}$.
    Since $\delta, \varepsilon>0$ are arbitrary constants, we obtained the desired conclusion.
\end{proof}

\begin{corollary}\label{cor: tail in strip}
    Let $\Sigma_t$ be an edge containing a left-going tail that converges to $\{x^2 = a_i\}$ as in Corollary~\ref{cor: sharp asympt line}. Then, for $t \leq t_0$,
    \begin{equation*}
        \Sigma_t\cap \{-\infty< x^1\leq R\sqrt{-t}\} \subset \{(x^1,x^2): \lvert x^2 - a_i \rvert \leq \exp(- \lvert t\rvert^{2\delta}) \}.
    \end{equation*}
\end{corollary}

\begin{proof}
    Combining Proposition~\ref{prop: tail in strip} and Corollary~\ref{cor: sharp asympt line} yields the desired result.    
\end{proof}

By a similar argument, we obtain the following result for fingers.

\begin{proposition}[finger in a strip]\label{prop: finger in strip}
    Let $\Gamma_t$ be a right-pointing finger of $M_t$ and let $p^1 \geq 0$ be a fixed value. Denote by $p_\pm(t):=(p^1,p_\pm^2(t))$ the intersection points of $\Gamma_t$ and $x^1 = p^1$ with $p_-^2(t) < p_+^2(t)$. Then,
    \begin{equation*}
        \Gamma_t \cap \big\{(x^1,x^2) \::\: x^1 \geq p^1 \big\}\subset \big\{(x^1,x^2)\::\: \inf_{s\leq t}\, p_-^2(s) \leq x^2 \leq \sup_{s\leq t}\, p_+^2(s)\big\}.
    \end{equation*}
   See Figure~\ref{fig: finger in strip}.
\end{proposition}

\begin{figure}[h]
    \begin{center}
            \includegraphics[width=0.75\linewidth]{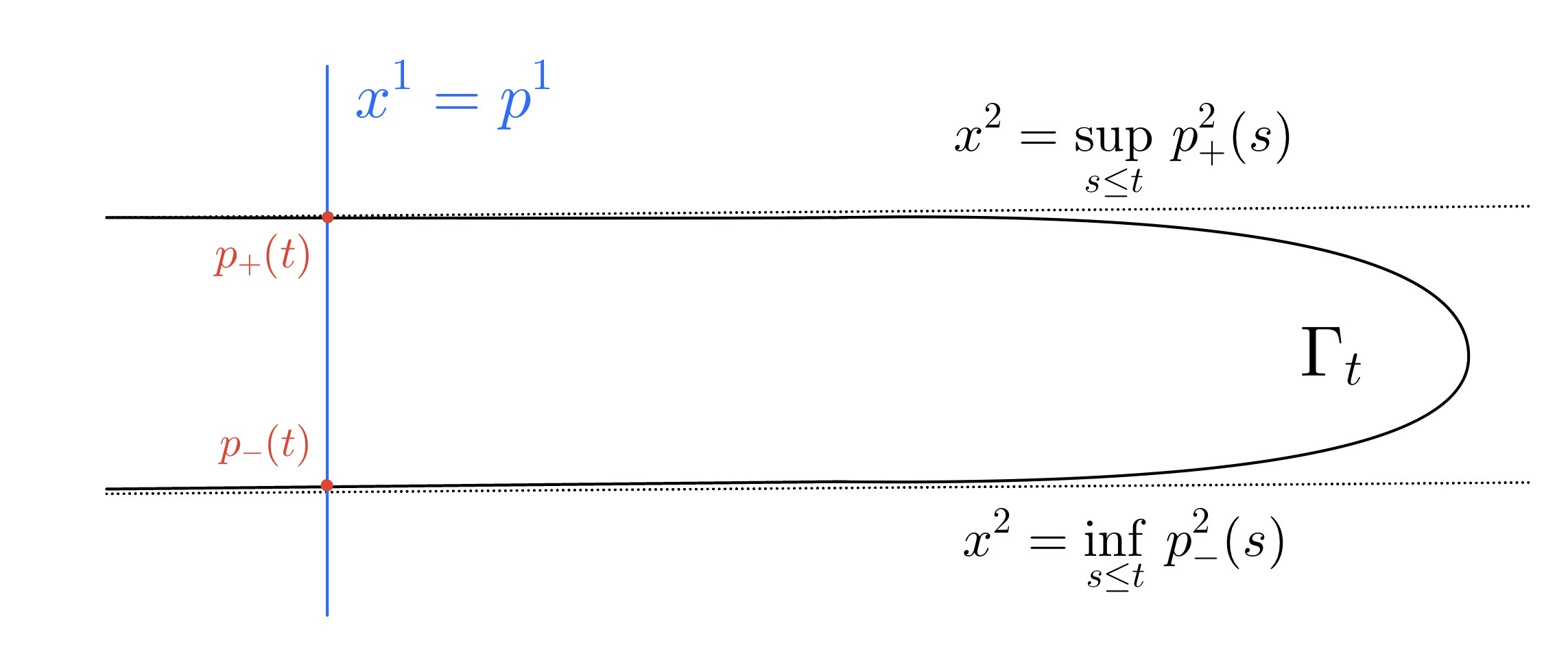}
    \end{center}
    \caption{Finger in strip.}
        \label{fig: finger in strip}
\end{figure}

\begin{proposition}[solution in strip]\label{prop: solution in strip}
    For all $t$,
    \begin{equation*}
        M_t \subset \R\times (-A, A),
    \end{equation*}
    where $A := 1 + \max\{\lvert a_1\rvert, \ldots, \lvert a_m\rvert\}$ and $a_i$'s are stated in Corollary~\ref{cor: sharp asympt line}.
\end{proposition}

\begin{proof}
    Applying Proposition~\ref{prop: tail in strip} and \ref{prop: finger in strip} to all fingers and tails of $M_t$, the result for $t\ll -1$ follows.
    Since $x^2 = \pm A$ are stationary solutions of CSF, by the avoidance principle, the result extends for all $t$.
\end{proof}

\bigskip

\subsection{Global geometry of the finger}\label{sec: glob.geom}

In this subsection, we always consider $\Gamma_t$ as a sub-flow of an \emph{extended} right-pointing finger of $M_t$, that is, $\Gamma_t$ is a union of two edges meeting at the unique sharp vertex $v(t)$ of $\Gamma_t$ for $t\leq t_{\vare}$ where $t_{\vare}$ is the $\vare$-trombone time (see Section 2 for the definition of $t_{\vare}$).
Additionally, we will always state the properties on the \emph{lower} edge $\Sigma_t$ of $\Gamma_t$; the corresponding properties on the upper edge $\Sigma_t'$ can be obtained by reflection across the $x^1$-axis.
Furthermore, as in Theorem~\ref{thm: sharp asympt line}, we assume that two sheets of $\Gamma_t$ converge to $x^2 = a_i, a_j$, respectively, where $a_i \leq a_j$; in particular, $\Sigma_t$ converges to $x^2 = a_i$.

Let $R\gg 1$, let $\be_1, \be_2$ be the standard orthonormal basis of $\R^2$, and let $\gamma(s, t)$ be an (orientation-preserving) arc-length parametrization of $\Gamma_t$ such that $\lvert \bt(s,t) - \be_1\rvert< \tfrac{1}{100}$ everywhere in $\Sigma_t \cap B(0, R\sqrt{-t})$ for all $t \leq t_R = -e^{-\tau_R}$ where $\tau_R$ is chosen to satisfy  \cite[Theorem 1.2]{CSSZ24} (and Theorem~\ref{thm: sharp asympt line}).
Recall that $\bt = \pa_s \gamma$ denotes the unit tangent vector and $\bn = J\bt$ denotes the unit normal vector of $\gamma$ where $J:\R^2 \to \R^2$ is the linear rotating map by $\tfrac{\pi}{2}$ counterclockwise.
Later, $t$ is always chosen so that $t \leq t_R$.

The angle function $\theta$ is a useful tool in the later discussion. 
Choose $\theta(s, t)$ to be the \emph{continuous} function measuring the angle of $\bt(s,t)$ deviating from $\be_1$ counterclockwise such that $\theta \approx 0$ in $\Sigma_t \cap B(0, R\sqrt{-t})$.
Recall that by \cite[Theorem 6.9]{CSSZ24} the angle difference between two knuckles of a finger is approximately $\pm\pi$ for $t \ll -1$.
Since $\Gamma_t$ is right-pointing, $\theta \approx \pi$ on the upper edge $\Sigma'_t$ of $\Gamma_t$ in $B(0, R\sqrt{-t})$.
Furthermore, the curvature at the tip (and at the sharp vertex \cite[Theorem 6.9]{CSSZ24}) of $\Gamma_t$ is \emph{positive}.

\bigskip
The following descriptions include all possible cases for $\Sigma_t$:
\begin{description}
    \item[(A)] $\Sigma_t$ is a curved segment connecting to an adjacent finger, by \cite[Lemma 6.13]{CSSZ24}, 
    \begin{description}
        \item[(A1)]\label{A1} $\Sigma_t$ has a unique inflection point with increasing curvature, or
        \item[(A2)] $\Sigma_t$ is convex ($\kappa>0$) and has a unique flat vertex.
    \end{description}
    \item[(B)] $\Sigma_t$ contains a tail and by \cite[Proposition 6.14]{CSSZ24} $\Sigma_t$ is convex ($\kappa>0$). On $\Sigma_t$, $\kappa$ is increasing in $s$ and decays to $0$ as $s \to -\infty$. 
\end{description}

\begin{proposition}\label{prop: anal.ang.kap.}
For $t\leq t_R$, the angle function $\theta(\cdot, t)$ and the curvature function $\kappa(\cdot, t)$ on $\Sigma_t$ have the following analytic significance:\\
\noindent {\bf (A1)}: $\theta(\cdot, t)$ is a positive convex function and $\kappa(\cdot, t)$ is an increasing function.\\
\noindent {\bf (A2)}: $\theta(\cdot, t)$ is an increasing function and $\kappa(\cdot, t)$ is a positive convex function.\\
\noindent {\bf (B)}: $\theta(s, t)$ and $\kappa(s, t)$ are both positive increasing functions in $s$ and decay to 0 as $s \to -\infty$.
\end{proposition}

\begin{proof}
    Positivity of $\theta$ in case (A1) will be shown in Proposition~\ref{prop: min angle at infl pt} later, and the rest of the statements in cases (A1) and (A2) follow directly from the above descriptions, setting of $\Sigma_t$, and the relation $\theta_s = \kappa$.

    In case (B), it is clear that $\theta(\cdot, t)$ and $\kappa(\cdot, t)$ are increasing, $\kappa(\cdot, t) >0$, and $\kappa(s,t) \to 0$ as $s\to -\infty$.
    It remains to show that $\theta(\cdot, t) > 0$ and $\theta(s,t) \to 0$ as $s\to -\infty$.
    Suppose that $\theta(s_0, t) \leq 0$ for some $s_0$. 
    Then for all $s\leq s_0$, $\theta(s,t) < 0$.
    Note that $\kappa$ remains very small in the tail region due to the estimate in the central region and the monotonicity of $\kappa$.
    Putting these together implies that Corollary~\ref{cor: tail in strip} will be violated for some $s\ll -1$. 
    Similary, $\theta(s,t) \to \theta_0 >0$ as $s\to -\infty$ would also contradict Corollary~\ref{cor: tail in strip}.
\end{proof}

\begin{proposition}\label{prop: min angle at infl pt}
    Assume (A1) holds.
    Then $\theta >0$ everywhere on $\Sigma_t$. 
    In particular, the curve segment of $\Sigma_t$ bounded between the minimum and maximum points of $x^1$-cooridnate is a graph over an interval in the $x^2$-axis.
\end{proposition}

\begin{figure}[h]
    \begin{center}
            \includegraphics[width=0.8\linewidth]{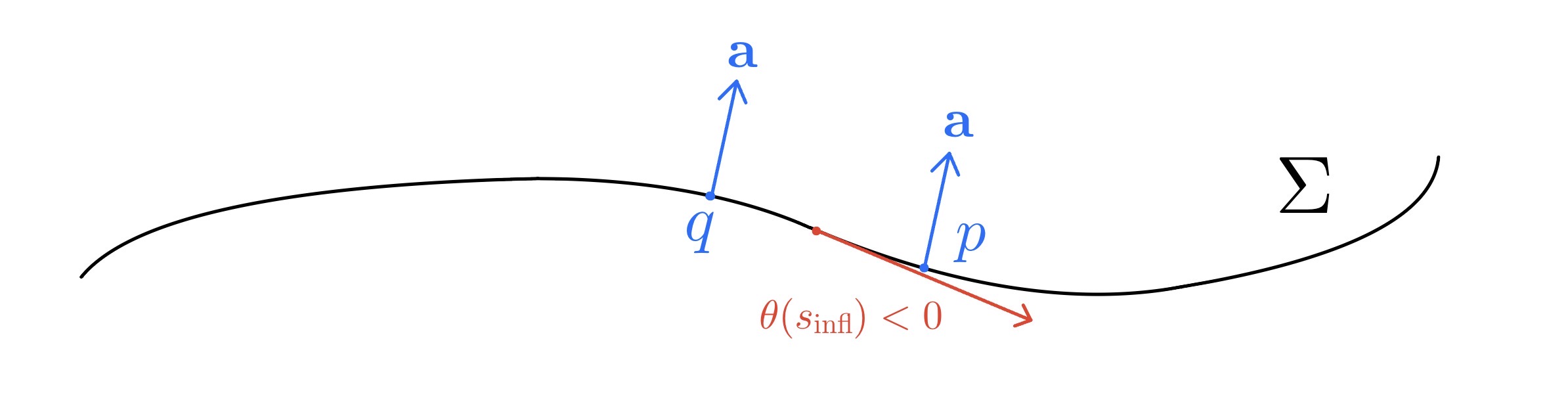}
    \end{center}
    \caption{An impossible sheet in case (A1) with $\theta(s_{\mathrm{infl}})<0$.}
        \label{fig: nongraphical}
\end{figure}

\begin{proof}
    By Proposition~\ref{prop: anal.ang.kap.}, $\theta(s, t)$ is a convex function in $s$ with a unique minimum at the inflection point $\gamma(s_{\mathrm{infl}}(t), t)$. 
    Suppose there exists a point on $\Sigma_{t_0}$ such that $\theta < 0$ for some $t_0$. 
    This immediately implies that $\theta(s_{\mathrm{infl}}(t_0), t_0) < 0$. 
    Pick a value $\theta_0$ such that $\max\{-\frac{\pi}{2}, \theta(s_{\mathrm{infl}}(t_0), t_0)\} < \theta_0 < 0$. Then there exist exactly two distinct points $(q^1(t_0), q^2(t_0)) := \gamma(q, t_0)$ and $(p^1(t_0), p^2(t_0)) := \gamma(p, t_0)$ with $q < s_{\mathrm{infl}}(t_0) < p$ satisfying $\theta(q, t_0) = \theta(p, t_0) = \theta_0$ as depicted in Figure~\ref{fig: nongraphical}.
    
    Choose the unit vector $\mathbf{a} = (-\sin \theta_0, \cos \theta_0)$ in the first quadrant and define the height function $E(s,t) = \langle \mathbf{a}, \gamma(s,t)\rangle$. 
    By the monotonicity of $\theta(\cdot, t_0)$ on either side of the inflection point, it follows that $p$ is a local minimum point of $E$ and $q$ is a local maximum point of $E$. Since $E$ satisfies the heat equation $\partial_t E = \partial_{ss} E$, Sturmian theory guarantees the existence of ancient paths $p(t)$ and $q(t)$ consisting of local minimum and local maximum points of $E$, respectively, such that $q(t) < p(t)$ for all $t\leq t_0$, $p(t_0) = p$ and $q(t_0) = q$.
    
    By the maximum principle, $E(p(t'), t') \leq E(p(t), t)$ and $E(q(t'), t') \geq E(q(t), t)$ for all $t' \leq t \leq t_0$. Applying Proposition~\ref{prop: solution in strip}, there exists $0< A' < \infty$ such that for all $t' \leq t \leq t_0$,
    \begin{align}
        q^1(t') &\geq q^1(t) - \cot\theta_0 \big( q^2(t) - q^2(t')\big) \geq q^1(t) - A', \label{eq: monotone q1}\\
        p^1(t') &\leq p^1(t) - \cot\theta_0 \big( p^2(t) - p^2(t')\big) \leq p^1(t) + A'. \label{eq: monotone p1}
    \end{align}

    According to Corollay~\ref{cor: sharp asympt line} and the setting of $\theta$, there exists $t_1 < t_R$ such that for all $t \leq t_1$:
    \begin{align}\label{eq: small angle}
        |\theta| \leq \exp(-\lvert t \rvert^{2\delta}) < |\theta_0| \quad \text{on} \quad \Sigma_{t} \cap B(0, R\sqrt{-t}).
    \end{align}
    We claim that $p^1(t_1) < 0 < q^1(t_1)$. Suppose, for the sake of contradiction, that $q^1(t_1) \leq 0$. Then by \eqref{eq: monotone q1}, $q^1(t_1) - A' \leq q^1(t) \leq 0$ for all $t \leq t_1$. When $t \ll t_1$ is chosen such that $R\sqrt{-t} > |q^1(t_1) + A|$, the fact that $\theta(q(t)) = \theta_0$ contradicts the small-angle estimate \eqref{eq: small angle}. A symmetric argument for $p^1$ shows that $p^1(t_1)$ cannot be non-negative.
    
    However, the condition $p^1(t_1) < 0 < q^1(t_1)$ is also impossible. Given that $q(t_1) < p(t_1)$, there exists a point within $\Sigma_{t_1} \cap B(0, R\sqrt{-t_1})$ satisfying $\bt \approx -\be_1$ that contradicts the initial setting of parametrization of $\gamma$. Thus, we conclude that $\theta \geq 0$ everywhere on $\Sigma_t$.

    Finally, we address the borderline case where $\theta = 0$ at the inflection point. Since $\theta$ also satisfies the parabolic equation $\partial_t \theta = \partial_{ss} \theta$, applying Sturm's theorem to $\theta$ around its local minimum (the inflection point) and invoking the strong maximum principle yields a contradiction. Consequently, $\theta > 0$ everywhere. In particular, the segment bounded by the extrema of the $x^1$-coordinate satisfies $\theta \in (0, \pi]$, and thus forms a graph over an interval in the $x^2$-axis.
\end{proof}

\begin{proposition}[optimal lower bound of curvature of a sharp vertex]\label{prop: vertex optimal speed}
    Let $\vare$ be sufficiently small. There exists $t_{\vare}\ll -1$ such that for all $t\leq t_{\vare}$
    \begin{equation}\label{eq: vertex optimal speed}
        \rvert\kappa\big(v(t)\big)\rvert \geq \frac{(1-2\vare)\pi}{\lvert a_i - a_j\rvert + 2\vare}.
    \end{equation}
    If $a_i \neq a_j$, then
    \begin{align}
        \liminf_{t\to -\infty} \,\lvert \kappa(v(t))\rvert \geq \frac{\pi}{\lvert a_i - a_j\rvert}.
    \end{align}
    If $a_i = a_j$, then the limit inferior is $+\infty$.
\end{proposition}

\begin{proof}
    Let $\vare >0$ be small. Without loss of generality, we can assume $a_i \leq a_{j}$.
    By Corollary~\ref{cor: sharp asympt line} and Proposition~\ref{prop: finger in strip}, there exists $t_{\vare} \ll -1$ such that for all $t<t_{\vare}$ 
    \begin{equation*}
        \Gamma_t \cap \{(x^1,x^2)\::\: x^1 \geq 0\}\subset \big\{(x^1,x^2) \::\: a_i-\vare\leq x^2 \leq a_j + \vare\big\}.
    \end{equation*}
    By \cite[Theorem 1.4]{CSSZ24}, by taking $t_{\vare}\ll -1$, for all $t<t_{\vare}$, $\cD_{\lvert \kappa(v(t))\rvert} \big(\cG-v(t)\big) \cap P(0, 1/\vare)$ is $\vare$-close in $C^2$ to a unit-speed translating grim reaper curve whose zero-time-slice has its tip at the origin. Consequently, sufficiently far away from the tip, the region enclosed by $\Gamma_t$ and the vertical line $x^1 = 1$ contains a ball of diameter $(1-2\vare)\pi \rvert\kappa\big(v(t)\big)\rvert^{-1}$.
    Putting all together, for all $t< t_{\vare}$, by comparing the diameter of the ball and the width of the strip, we obtain \eqref{eq: vertex optimal speed}.
    By letting $\vare \to 0^+$, we complete the proof.
\end{proof}

\bigskip

Recall that a model right-pointing grim reaper that has unit curvature at the tip at the origin can be parametrized by its angle $\varphi$ deviated from the angle of the tip via the map
\begin{align*}
    \gamma(\varphi) = (\log \cos \varphi,\varphi ) \quad \text{for $-\tfrac{\pi}{2} <\varphi < \tfrac{\pi}{2}$.}
\end{align*}
The Euclidean distance between $\gamma(\varphi)$ and the tip $(0,0)$ is given by $d^2 := \varphi^2 + (\log \cos \varphi)^2$. 
Thus, if $d\geq 200 \pi$
\begin{align*}
     e^{-d}\leq \cos \varphi = e^{-\sqrt{d^2 - \varphi^2}} \leq e^{-\sqrt{d^2 - \pi^2/4}} \leq e^{-\frac{99}{100}d},
\end{align*}
and therefore, using $\cos \lvert \varphi\rvert = \sin(\tfrac{\pi}{2} - \lvert \varphi\vert)$ and $\sin x\geq \tfrac12 x$ for $0\leq x\ll 1$,
\begin{equation}\label{eq: angle.decay.gr}
    e^{-d}\leq (\tfrac{\pi}{2}- \lvert \varphi\rvert) \leq e^{-\frac{98}{100}d}.
\end{equation}

\bigskip
The following proposition asserts that the vertex of a right-pointing finger is moving almost in $\be_1$-direction for $t\ll -1$. 
\begin{lemma}[travelling direction and horizontal speed of a vertex]\label{lem: direction of vertex}
    Let $\vare >0$ be sufficiently small. 
    There exists $t_{\vare} \ll -1$ such that for all $t\leq t_{\vare}$
    \begin{equation}\label{eq: direction of vertex}
        \lvert \theta(v(t),t) - \tfrac{\pi}{2}\rvert < \vare,
    \end{equation}
    and 
    \begin{align*}
        \tfrac{\pa}{\pa t} \langle v(t), \be_1 \rangle 
        \leq -\frac{(1- 4\vare)\pi}{\lvert a_i - a_j\rvert + 2\vare}
        \leq -\frac{3}{2A},
    \end{align*}
    where $A:= 1+ \max_i \, \lvert a_i\rvert $.
\end{lemma}

\begin{proof}
    Let $\vare>0$ be sufficiently small such that $\vare^{\frac12} \lvert\log \vare\rvert < \tfrac{1}{100}$.
    By applying reflection across the $x^1$-axis, it suffices to show that for $t \ll -1$
    \begin{equation}\label{eq: direction of vertex1}
        \theta(v(t),t) -\tfrac{\pi}{2} \geq  -\vare.
    \end{equation}
    By \cite[Theorem 1.4]{CSSZ24}, there exists $t_{\vare} \ll -1$ such that for $t \leq t_{\vare}$, $\Gamma_{t}  \cap B(v(t), 100\vare^{-1}\lvert \kappa(v(t))\rvert^{-1})$ is $\tfrac{\vare}{100}$-close in $C^2$ to a grim reaper curve of width $\pi\lvert \kappa(v(t))\rvert^{-1}$.
    Let $q_t$ denote the intersection point in the lower sheet $\Sigma_t \cap \pa B(v(t), 50\vare^{-\frac12}\lvert \kappa(v(t))\rvert^{-1})$.
    By the estimate \eqref{eq: angle.decay.gr}, 
    \begin{equation}\label{eq: angle diff of g.r.}
        \theta(q_t, t) - \tfrac{\vare}{25}  \leq \theta(v(t), t) - \tfrac{\pi}{2} \leq \theta(q_t, t) + \tfrac{\vare}{25},
    \end{equation}
    where the contribution of angle difference between $v(t)$ and $q_t$ from the model grim reaper curve is less than $\exp(-49\vare^{-\frac12})$ due to \eqref{eq: angle.decay.gr}, significantly less than $\tfrac{\vare}{100}$ due to the assumption on smallness of $\vare$ in the beginning.
    It follows from Proposition~\ref{prop: vertex optimal speed} that by taking $t_\vare\leq t_1$, for all $t\leq t_\vare$ the scale 
    \begin{equation}\label{eq: curv.scale}
        \lvert \kappa(v(t))\rvert^{-1} \leq (\lvert a_i - a_j \rvert + \vare) \leq 2A.
    \end{equation}
    Thus, for $\langle v(t),\be_1\rangle \gg 1$, $q_t$ lies in $\{(x^1, x^2)\,:\, x^1 \geq 0\}$.

    If (A1) holds, then Proposition~\ref{prop: min angle at infl pt} implies that $\theta(q_t, t)\geq 0$
    and thus \eqref{eq: angle diff of g.r.} gives \eqref{eq: direction of vertex1}.
    By taking $t_{\vare}\ll -1$ and by \cite[Theorem 1.2]{CSSZ24}, $\lvert \theta\rvert < \tfrac{\vare}{100}$ in the graphical region $\Sigma_t \cap B_{10}(0)$.
    If (A2) or (B) holds, then $\theta$ is increasing in $s$ and thus $\theta(q_t, t) > -\tfrac{\vare}{100}$ on $\Sigma_t \cap \{x^1 \geq 0\}$, implying \eqref{eq: direction of vertex1}.

    Finally, it follows from the property of an $\tfrac{\vare}{100}$-grim reaper, Proposition~\ref{prop: vertex optimal speed}, and \eqref{eq: curv.scale} that for $t\leq t_{\vare}$
    \begin{align*}
        \tfrac{\pa}{\pa t} \langle v(t), \be_1 \rangle \leq (1- \tfrac{\vare}{100})\langle \kappa(v(t))\bn, \be_1\rangle 
        \leq -\tfrac{(1- 4\vare)\pi}{\lvert a_i - a_j\rvert + 2\vare}
        \leq -\tfrac{3}{2A},
    \end{align*}
    provided that $\vare$ is sufficiently small.
\end{proof}

\bigskip
\begin{proposition}[interior gradient estimate]\label{prop: ang.away.tip}
    Let $\vare >0$ be sufficiently small. 
    There exists $t_{\vare} \ll -1$ such that for all $t\leq t_{\vare}$, if one of the following cases holds:
    \begin{description}
        \item[(A)] $\Sigma_t$ is a compact edge with vertices $v_1(t) = \big(v_1^1(t), v_1^2(t))$ and $v_2(t)= \big(v_2^1(t), v_2^2(t)\big)$ such that $v_1^1(t) < v_2^2(t)$, and $p = (p^1, p^2)\in \Sigma_t$ with 
        \begin{equation}\label{eq: range1}
            v_1^1(t) + 50\vare^{-\frac12}\lvert \kappa\big(v_1(t)\big)\rvert^{-1} 
            \leq p^1 \leq v_2^1(t) - 50\vare^{-\frac12} \lvert \kappa\big(v_2(t)\big)\rvert^{-1},
        \end{equation}

        \item[(B)] $\Sigma_t$ is a noncompact edge with only one vertex $v(t) = \big( v^1(t), v^2(t)\big)$ and $p = (p^1, p^2) \in \Sigma_t$ with 
        \begin{equation}\label{eq: range2}
            p^1 \leq  v^1(t) - 50\vare^{-\frac12} \lvert \kappa\big(v(t)\big)\rvert^{-1},
        \end{equation}
    \end{description}
    then
    \begin{equation*}
        \lvert \theta(p,t) \rvert \leq 2\vare.
    \end{equation*}
    In particular, $\Sigma_t$ is a graph of $U(x,t)$ with $\lvert U_x\rvert \leq 4\vare$ for $x$ in the domain of \eqref{eq: range1} or \eqref{eq: range2} .
\end{proposition}

\begin{proof}
    We will follow the same setting in the proof of Lemma~\ref{lem: direction of vertex} and freeze $t\leq t_{\vare}$. 

    Suppose that {\bf Case (A)} holds.
    Let $\gamma:[a, b]\rightarrow \R^2$ be an arclength parametrization of $\Sigma_t$ with $\gamma(a) = v_1(t)$, $\gamma(0)$ the intersection point of $\Sigma_t$ with the $x^2$-axis, and $\gamma(b) = v_2(t)$.
    By Lemma~\ref{lem: direction of vertex} and by taking $t_{\vare}\ll -1$, we assume that for $i = 1,2$, $M_{t}  \cap B(v_i(t), 100\vare^{-1}\lvert \kappa(v_i(t))\rvert^{-1})$ is $\tfrac{\vare}{100}$-close in $C^2$ to a grim reaper curve of width $\pi\lvert \kappa(v_i(t))\rvert^{-1}$, with a small tilted angle ($<\vare$) between the axis of symmetry and the $x^1$-axis. 
    Let $q =(q^1, q^2)= \gamma(c)$, with $0 < c < b$, denote the point in $\Sigma_t$ with $q^1 = v_2^1(t) - 50 \vare^{-\frac12}\lvert \kappa(v_2(t))\rvert^{-1}$. 
    In view of \eqref{eq: angle diff of g.r.}, $\lvert \theta(c)\vert\leq 2\vare$.
    By reflection, it suffices to show that $\lvert\theta(s)\rvert\leq 2\vare$ for $s \in [0, c) $.

    \noindent{\bf Case (A1)} : By Proposition~\ref{prop: anal.ang.kap.}, $\theta$ is a positive convex function.
    Then \eqref{eq: direction of vertex} and \eqref{eq: angle diff of g.r.} imply that for any $s\in [0, c)$
    \begin{equation*}
        0 < \theta(s) \leq \max\, \{\theta(0), \theta(c)\} \leq 2\vare.
    \end{equation*}
    \noindent{\bf Case (A2)} : By Proposition~\ref{prop: anal.ang.kap.}, $\theta$ is an increasing function.
    For $s\in [0, c)$, 
    \begin{equation*}
        -\tfrac{\vare}{100} < \theta(0) \le \theta(s) \leq \theta(c) < 2\vare. 
    \end{equation*}

    Suppose that {\bf Case (B)} holds. 
    Let $\gamma: (-\infty, b]\rightarrow \R^2$ be an arclength parametrization of $\Sigma_t$ with  $\gamma(b) = v(t)$ and $\gamma(c) = q$.
    By Proposition~\ref{prop: anal.ang.kap.}, $\theta$ is a positive increasing function.
    Therefore, for $s\in (-\infty, c)$
    \begin{equation*}
        0 < \theta(s) \leq \theta(c) < 2\vare.
    \end{equation*}
   
\end{proof}

\begin{proposition}[curvature estimate]\label{prop: curv.away.tip}
    Let $\vare>0$ be sufficiently small and let $t_{\vare}\ll -1$ and $L := 100\vare^{-\frac12}A$.
    For any $t\leq t_{\vare}$ and any $D > 0$, if one of the following cases holds:
    \begin{description}
        \item[(A)] $\Sigma_t$ is a compact edge with vertices $v_1(t) = \big(v_1^1(t), v_1^2(t))$ and $v_2(t)= \big(v_2^1(t), v_2^2(t)\big)$ such that $v_1^1(t) < v_2^2(t)$, and $p = (p^1, p^2)\in \Sigma_t$ with 
        \begin{equation*}
            v_1^1(t) + L + D 
            \leq p^1 \leq v_2^1(t) - L - D,
        \end{equation*}

        \item[(B)] $\Sigma_t$ is a noncompact edge with only one vertex $v(t) = \big( v^1(t), v^2(t)\big)$ and $p = (p^1, p^2) \in \Sigma_t$ with 
        \begin{equation*}
            p^1 \leq  v^1(t) - L -D,
        \end{equation*}
    \end{description}
    then
    \begin{align*}
        \lvert \kappa(p, t)\rvert \leq \frac{4\vare}{D}.
    \end{align*}
    
\end{proposition}

\begin{proof}
    Let $t\leq t_{\vare}$ be fixed. 
    Suppose that {\bf Case (A)} holds.
    Let $\gamma:[a, b]\rightarrow \R^2$ be an arclength parametrization of $\Sigma_t$ with $\gamma(a) = v_1(t)$ and $\gamma(b) = v_2(t)$.
    By Proposition \ref{prop: ang.away.tip} and \eqref{eq: curv.scale}, for any $s\in (a+L, b-L)$, $\lvert\theta(s)\rvert < 2\vare$. 
    The mean value theorem implies that there exist $c\in (a+L, a+L+D)$ and $d\in (b-L-D, b-L)$ such that $\lvert \kappa(c)\rvert < 4\vare/D$ and $\lvert \kappa(d)\rvert < 4\vare/D$.
    
    \noindent{\bf Case (A1)} : By Proposition~\ref{prop: anal.ang.kap.}, $\kappa$ is increasing.
    It follows that for $s\in (a+ L + D, b-L-D)$,
    \begin{equation*}
        -\tfrac{4\vare}{D} < \kappa(c) \le \kappa(s) \le \kappa(d) < \tfrac{4\vare}{D}.
    \end{equation*}

    \noindent{\bf Case (A2)} : By Proposition~\ref{prop: anal.ang.kap.}, $\kappa$ is a positive convex function.
    Thus,
    \begin{equation*}
        0< \kappa(s) \leq \max\:\{\kappa(c), \kappa(d)\} < \tfrac{4\vare}{D}.
    \end{equation*}

    \bigskip
    Suppose that {\bf Case (B)} holds: $\Sigma_t$ is a noncompact edge with only one vertex $v(t)$. 
    Let $\gamma: (-\infty, b]\rightarrow \R^2$ be an arclength parametrization of $\Sigma_t$ with $\gamma(b) = v(t)$.
    By Proposition \ref{prop: ang.away.tip}, \eqref{eq: curv.scale}, and the Mean Value Theorem, we may argue as in Case (A) to find some $d\in (b-L-D, b-L)$ satisfying $\lvert \kappa(d)\rvert < 4\vare/D$.
    By Proposition~\ref{prop: anal.ang.kap.}, $\kappa$ is a positive increasing function and thus for $s < b-L-D$,
    \begin{equation*}
        0< \kappa(s) \le \kappa(d) < \tfrac{4\vare}{D}.
    \end{equation*}
\end{proof}

\section{Nondegeneracy of finger and sharp tip asymptotic}

\subsection{Exponential convergence of a sheet}
In this section, we will show the nondegeneracy of the finger and the sharp asymptotic of the vertex. 
The key tool is the following quantitative convergence theorem.
\begin{theorem}[quantitative exponential convergence of edge]\label{thm: quan.conv}
Let $\vare_0\in (0,\tfrac{1}{100})$, $A>0$ be arbitrary.
There exist constants $D_0\gg 1$ and $\beta \in (0,\tfrac{1}{100})$ depending only on $A$ but not on $\vare_0$ with the following significance: 
Let $D \geq D_0$,
\begin{equation*}
    \varrho(t) := D + \tfrac{1}{2A}(-t)\quad \mbox{for }\, t\leq 0,
\end{equation*}
and let $\Sigma_t$ be a sheet of CSF that can be expressed as a graph of $U$, where $U:[-\varrho(t),\varrho(t)]\times (-\infty,0]\to \mathbb{R}$ satisfies for all $t\leq 0$
\begin{equation*}
    \lVert U(\cdot, t)\rVert_{L^\infty([-\varrho(t), \varrho(t)])} < A, \quad \|U_x(\cdot,t)\|_{L^\infty([-\varrho(t),\varrho(t)])} \leq \varepsilon_0,
\end{equation*}
and 
\begin{equation*}
    \|U_{xx}(\cdot,t)\|_{L^\infty([-\varrho(t),\varrho(t)])} \leq \frac{\varepsilon_0}{\varrho(t)}.
\end{equation*}
Assume that $U\to 0$ in $C^\infty_{\text{loc}}$ as $t \to -\infty$. 
Then, for all $t\leq 0$
\begin{equation*}
    |U(0,t)|\leq e^{-\beta D}.
\end{equation*}
\end{theorem}

\begin{proof}
    By shifting in time, we consider 
    \begin{equation*}
        V(x,t) := U(x, t + 2AD):\big[-\frac{1}{2A}\lvert t\rvert, \frac{1}{2A}\lvert t\rvert\big]\times (-\infty,-2AD]
    \end{equation*}
    and the rescaled flow 
    \begin{equation*}
        v(y,\tau)=e^{\tau/2}U(e^{-\tau/2} y,-e^{-\tau}) \quad \text{for} \quad \lvert y\rvert \leq e^{-\tau/2}/(2A),\quad \tau \leq -\log (2AD).
    \end{equation*}
 
    Choose the graphical radius for rescaled flow to be $\rho=e^{-\tau/2}/(4A)$. 
    Converting the estimates in the assumption, we find that $v$ satisfies \eqref{hyp: bounds} and \eqref{hyp: decay.cond} with $K_0 = 1$. 
    Clearly, $\rho$ satisfies \eqref{hyp: bdd.log.d.rho.1} and \eqref{hyp: bdd.log.d.rho.2} with (optimal) $\mu = \tfrac12$.
    Choose $D_0 = 8AB^2$, where $B = B(A,K_0, \mu)$ with $K_0 = 1$, $\mu = \tfrac12$ is stated in \eqref{hyp: ini.cond}.
    By Theorem~\ref{thm: general asymp line}, there exist $a\in \R$ and a numerical constant $C$ such that for all $\tau \leq -\log (2AD)$
    \begin{equation*}
        \|e^{-\tau/2}v(\cdot,\tau) - a\|^2_{L^2([-10,10])} \leq C(1+ \lvert a \rvert)e^{- \rho^2/49}.
    \end{equation*}
    Since $V\to 0$ in $C^\infty_{loc}$, we have $a = 0$.
    By slightly modifying the last paragraph in the proof of Theorem~\ref{thm: general asymp line}, there exist numerical constants $C$ and $\beta$ such that for all $\tau \leq -\log (2AL)$
    \begin{equation*}
        \|e^{-\tau/2}v(\cdot,\tau) \|^2_{L^\infty([-5,5])} \leq Ce^{-\beta \rho(\tau)^2} \leq C e^{-\frac{\beta D}{8A}}.
    \end{equation*}
    After translating it back to the original coordinates and reselecting a smaller $\beta$ and a larger $D_0$ that depend only on $A$, we obtain the desired estimate.
\end{proof}

\begin{theorem}[quantitative height estimate]\label{thm: exp.conv.away.tip}
    Let $\cM$ be as stated in Theorem~\ref{thm: sharp asympt line} and let $A = 1 + \max_i \: \lvert a_i\rvert$.
    Let $\Sigma_t$ be an edge of a right-pointing finger $\Gamma_t$ of $M_t$.
    Suppose that $\Sigma_t$ converges to the line $\{x^2 = a_i\}$ in $C^\infty_{\mathrm{loc}}$ as $t\to -\infty$.
    Let $\vare>0$ be sufficiently small, and let $t_{\vare}$, $L = 100\vare^{-\frac12}A$, and $U(x,t)$ be as stated in Lemma~\ref{lem: direction of vertex}, Proposition~\ref{prop: ang.away.tip} and ~\ref{prop: curv.away.tip}.
    There exist $D_0>0$ and $\beta \in (0, \tfrac{1}{100})$ depending only on $A$ with the following significance:
    For any $t_0\leq t_{\vare}$,
    \begin{description}
        \item[(A)] if $\Sigma_t$ is a compact edge with vertices $v_i(t)= (v_i^1(t), v_i^2(t))$, $i = 1,2$, with $v_1^1(t) < v_2^1(t)$, and $x\in \R$ satisfies 
        \begin{equation*}
            v_1^1(t_0) + L + 2D_0\leq x \leq v_2^1(t_0) - L - 2D_0,
        \end{equation*}
        then for all $t\leq t_0$, $\Sigma_t$ satisfies
        \begin{equation*}
            \lvert U(x,t) - a_i\rvert < \exp\big(-\beta(\min_{i=1,2}\, \lvert x-v_i^1(t_0)\rvert - L)\big).
        \end{equation*}
        \item[(B)] if $\Sigma_t$ is a noncompact edge with only one vertex $v(t) = (v^1(t), v^2(t))$ and $x\in \R$ satisfies 
        $$x < v^1(t_0) -  L - 2D_0,$$
        then for all $t\leq t_0$, $\Sigma_t$ satisfies
        \begin{equation*}
            \lvert U(x,t) - a_i\rvert < \exp\big(-\beta(\lvert x - v^1(t_0)\rvert - L)\big).
        \end{equation*}
    \end{description}
\end{theorem}

\begin{proof}
    The proofs of cases (A) and (B) are similar. We will only present the proof for case (A).
    Let $\vare\in (0, 10^{-3})$ be sufficiently small to be determined, let $D_0$ to be as stated in Theorem~\ref{thm: quan.conv}, and fix $t_0 \leq t_{\vare}$ and $x$ in the stated domain.
    Define
    \begin{equation*}
        \varrho(t) = \tfrac12 \big( \min_{i=1,2}\, \lvert x-v_i^1(t_0)\rvert - L\big) + \tfrac{1}{2A}(t_0 - t).
    \end{equation*}
    By assumption, $\varrho(t_0) \geq D_0$. 
    In light of Lemma~\ref{lem: direction of vertex}, for all $t \leq t_0$, $i = 1,2$,
    \begin{equation*}
        \lvert x-v_i^1(t)\rvert \geq \lvert x-v_i^1(t_0)\rvert + \tfrac{3}{2A}(t_0-t) \geq 2\varrho(t) +L.
    \end{equation*}
    Therefore, for all $t\leq t_{\vare}$, by Proposition~\ref{prop: solution in strip} and ~\ref{prop: ang.away.tip}, $\Sigma_t \cap \{(x^1, x^2) \in \R^2:\lvert x^1 - x \rvert \leq \varrho(t)\}$ is a graph of $U(x, t)$ with $\|U(\cdot,t)\|_{L^\infty([x-\varrho(t), x+ \varrho(t)])} < A$  and $\|U_x(\cdot,t)\|_{L^\infty([x-\varrho(t), x+ \varrho(t)])} < 4\vare$.
    Furthermore, by Proposition~\ref{prop: curv.away.tip}, $\|U_{xx}(\cdot,t)\|_{L^\infty([x-\varrho(t), x+ \varrho(t)])} < 8\vare/\varrho(t)$.
    Up to a proper translation in spacetime, Theorem~\ref{thm: quan.conv} implies that there exists $\beta>0$ depending only on $A$ such that for all $t\leq t_0$
    \begin{equation*}
        \lvert U(x, t) - a_i\rvert < \exp\big(-\tfrac{\beta}{2} \big( \min_{i=1,2}\, \lvert x-v_i^1(t_0)\rvert - L\big)\big).
    \end{equation*}
    Reselecting a smaller $\beta$ then gives the desired estimate.
\end{proof}

\subsection{Nondegeneracy of finger}

\begin{theorem}\label{thm: non-degenerate finger}
    Each finger is non-degenerate at $t = -\infty$; namely, the two horizontal asymptotes of each finger are separated by a positive distance.
\end{theorem}

\begin{remark}
In fact, we will prove a stronger theorem that each horizontal asymptote in Theorem~\ref{thm: sharp asympt line} has multiplicity one; namely, $a_i\neq a_{j}$ whenever $i\neq j$.
\end{remark}

\begin{proof}
    For the purpose of contradiction, we assume that there exists a degenerate (extended) finger $\Gamma_t$ of $M_t$ with multiplicity-two horizontal limiting line $x^2=a$. 
    For the sake of brevity, by translation, we assume that $a = 0$; by rotation and reflection, we assume the finger $\Gamma_t$ is right-pointing.
    Let $v(t) = (v^1(t), v^2(t))$ denote the unique sharp vertex of $\Gamma_t$.
    Let $\vare>0$ be sufficiently small, and let $t_{\vare}\ll -1$ and $L$ be as stated in Theorem~\ref{thm: exp.conv.away.tip}.
    For any $t\leq t_{\vare}$, applying Theorem~\ref{thm: exp.conv.away.tip} with $p^1 = \tfrac{1}{2}v^1(t) \gg L$ and a slightly smaller $\beta$, together with Proposition~\ref{prop: finger in strip}, we obtain
    \begin{equation*}
        \Gamma_t\cap \big\{(x^1, x^2)\in \R^2\::\:x^1\geq \tfrac{1}{2}v^1(t)\big\} \subset \big\{(x^1, x^2)\in \R^2\::\: \lvert x^2\rvert \leq e^{-\beta v^1(t)}\big\}.
    \end{equation*}
    Here, we absorb a factor $\tfrac12$ by reselecting $\beta >0$ for simplicity.
    Using the ball comparison argument in the proof of Proposition~\ref{prop: vertex optimal speed}, we improve the curvature lower bound: for $t\leq t_{\vare}$,
    \begin{equation*}
        \lvert\kappa(v(t))\rvert \geq  (1-2\vare)\frac{\pi}{2} e^{\beta v^1(t)}.
    \end{equation*}
    
    By Lemma~\ref{lem: direction of vertex} and modifying the last paragraph of its proof, for $t\leq t_{\vare}$,
    \begin{equation*}
        \tfrac{\pa}{\pa t}v^1(t) \leq (1- \tfrac{\vare}{100})\langle \kappa(v(t))\bn, \be_1\rangle \leq -(1-\vare)(1-2\vare)\frac{\pi}{2}  e^{\beta v^1(t)} \leq -e^{\beta v^1(t)},
    \end{equation*}
    and hence
    \begin{equation}\label{eq: super.fast.tip}
        \tfrac{\pa}{\pa t}e^{-\beta v^1(t)}= -\beta e^{-\beta v^1(t)} \tfrac{\pa}{\pa t}v^1(t) \geq \beta.
    \end{equation}
    Let $t_0\leq t_{\vare}$ be fixed. Then for any $t< t_0$, integrating \eqref{eq: super.fast.tip} over the interval $[t, t_0]$ yields
    \begin{equation*}
        v^1(t) \geq -\tfrac{1}{\beta}\log\big( e^{-\beta v^1(t_0)} - \beta(t_0 - t)\big).
    \end{equation*}
    Observe that since $\beta > 0$ and $v^1(t_0) \gg 1$, the argument of the logarithm on the right-hand side vanishes after a finite decrease in $t$. 
    Consequently, $v^1(t)$ diverges to infinity in finite backward time. This implies that the finger $\Gamma_t$ becomes disconnected or folded, contradicting the connectedness or embeddedness assumption of $M_t$, respectively.
\end{proof}

\bigskip
\bigskip

\subsection{Sharp tip asymptotics}

In this section, we show that each finger converges to a translating grim reaper. 
Without loss of generality, we assume that $\Gamma_t$ is a right-pointing finger with sharp vertex $v(t)$, which is asymptotic to lines $x^2 = a_i$ and $x^2 = a_j$ as $t \to -\infty$ with strict inequality $a_i < a_j$ due to Theorem~\ref{thm: non-degenerate finger}.

\bigskip

\begin{proposition}\label{prop: tip_local_convergence}
For all sufficiently small $\vare > 0$, there exists $t_{\vare}\ll -1$ such that for all $t\leq t_{\vare}$, $\Gamma_t  \cap B_{1/\vare}(v(t))$ is $\vare$-close to the grim reaper curve bounded between the asymptotes $x^2=a_i$ and $x^2 = a_j$ with tip at $v(t)$ in $C^k$-topology.
\end{proposition}

\begin{proof}
Let $0< \delta\ll \vare\ll \lvert a_i - a_j\rvert$.
By Lemma~\ref{lem: direction of vertex} and its proof, for all sufficiently small $\delta>0$ there exists $T_\delta \ll -1$ such that for all $t\leq T_\delta$, $(\Gamma_t - v(t))\cap B(0, 100\delta^{-1}\lvert \kappa(v(t))\rvert^{-1})$ is $\tfrac{\delta}{100}$-close in curvature-scale-weighted $C^{\lfloor 100/\delta\rfloor}$-topology to a grim reaper curve of width $\pi\lvert \kappa(v(t))\rvert^{-1}$ with a small tilted angle ($<\delta$) between the axis of symmetry and the $x^1$-axis.

Let $q^1 = v^1(t) - 50 \delta^{-\frac12}\lvert \kappa(v(t))\rvert^{-1}$.
By Proposition~\ref{prop: ang.away.tip}, the lower and upper edges of $\Gamma_t \cap \{(x^1,x^2)\in \R^2\::\:0\leq x^1\leq q^1\}$ can be expressed as the graphs of $U^1(x,t) < U^2(x,t)$ for $0\leq x\leq q^1$, respectively, with $\lvert U_x^i\rvert\leq 4\delta$, $i = 1,2$. 
Furthermore, by the asymptotics to grim reaper curve in the previous paragraph, 
\begin{equation}\label{eq: D.U.ub}
    \lvert U^1(q^1,t) - U^2(q^1,t)\rvert \leq \sec \delta\cdot(\pi + \tfrac{\delta}{25}) \lvert \kappa(v(t))\rvert^{-1} \leq (1+ \delta) \pi \lvert \kappa(v(t))\rvert^{-1}.
\end{equation}

Freeze $t\leq T_\delta$. 
Let $A$, $D_0$, and $\beta$ be stated as in Theorem~\ref{thm: exp.conv.away.tip} with $\delta$ in place of $\vare$. 
Take $D \geq D_0$ such that $e^{-\beta D} \leq \delta$. 
Since $D_0$ is a numeric constant, by taking $\delta$ sufficiently small, we may simply assume that $D = -\beta^{-1}\log \delta$. 
Then, by Theorem~\ref{thm: exp.conv.away.tip}, for $p^1 = v^1(t) - 100\delta^{-\frac12}A- 2D$, 
\begin{equation}\label{eq: U.p}
    \lvert U^1(p^1, t) - a_i\rvert \leq \delta, \quad \lvert U^2(p^1,t) - a_j\rvert \leq \delta.
\end{equation}
Additionally, by taking $\delta$ sufficiently small so that $\rvert \log \delta \rvert \delta^{\frac12}\ll \beta A$, it follows from the derivative estimate that
\begin{equation}\label{eq: U.q}
    \lvert U^i(p^1, t) - U^i(q^1,t)\rvert \leq  \lvert p^1 - q^1\rvert \cdot \lVert U_x^i\rVert_{L^\infty([p^1,q^1])} \leq (2D+50\delta^{-\frac12}A)\cdot(4\delta) \leq 201 \delta^{\frac12}.
\end{equation}
Combining \eqref{eq: U.p} and \eqref{eq: U.q} yields that with $\delta\leq A\delta^{\frac12}$
\begin{equation}\label{eq: D.U.lb}
    \lvert U^1(q^1, t) -U^2(q^1,t)\rvert \geq \lvert a_i - a_j\rvert -  500 A\delta^{\frac12}.
\end{equation}
Recalling $\lvert a_i - a_j \rvert >0$, by taking $0<\delta \ll \vare \ll \lvert a_i - a_j \rvert$, combining \eqref{eq: D.U.ub} and \eqref{eq: D.U.lb} gives
\begin{equation*}
    \pi\lvert \kappa(v(t))\rvert^{-1} \geq (1- \tfrac{\vare}{100}) \lvert a_i - a_j\rvert.
\end{equation*}
Finally, putting this estimate for width back into the asymptotics to grim reaper in the first paragraph yields the desired result.

\end{proof}

\bigskip

Let $F(t)$ denote the \emph{finger region} in the right half-plane $\R_+\times \R$ enclosed by $\Gamma_t$ and the $x^2$-axis, and let $\lvert F(t)\rvert$ denote its area.
\begin{proposition}[area of finger region]\label{prop: area asymp}
    There exist constants $C_0\in \R$ and $\beta>0$ such that as $t \ll -1$, the area of $F(t)$ satisfies
    \begin{equation}\label{eq: asymp.expansion.1}
        \lvert F(t)\rvert = -\pi t + C_0  + O(e^{\beta t}).
    \end{equation}
\end{proposition}

\begin{proof}
    We follow the setting in Subsection~\ref{sec: glob.geom}. 
    Let $\gamma(p(t), t)$ and $\gamma(q(t), t)$ denote the $x^2$-intercept of the lower and upper sheets of $\Gamma_t$, respectively. 
    By Lemma~\ref{lem: direction of vertex}, the $x^1$-coordinates of all vertices of $M_t$ are at least $\tfrac{3}{2A}\lvert t\rvert + O(1)$ away from 0 for $t\ll -1$.
    Then Theorem~\ref{thm: exp.conv.away.tip} implies that there exists $\beta\in (0, \tfrac{1}{100})$ such that the profile functions $U^1, U^2$ of lower and upper sheets of $\Gamma_t$ in a neighborhood around the origin satisfy for $t\ll -1$
    \begin{equation*}
        \lVert U^1(\cdot, t) - a_i\rVert_{L^\infty([-10,10])}< e^{\beta t},\quad  \lVert U^2(\cdot, t) - a_j\rVert_{L^\infty([-10,10])}  < e^{\beta t},
    \end{equation*}
    where we use the fact that $-3t/2A\gg L + 2D_0$ for $t\ll -1$ to simplify the exponential function and use $\beta$ to absorb $1/A$.
    Combining this with the standard interpolation inequality \cite[Lemma B.1]{CM15} and the uniform $C^2$-bound from Proposition~\ref{prop: ang.away.tip} and \ref{prop: curv.away.tip}, there exists (smaller) $\beta\in (0, \tfrac{1}{100})$ such that $\lVert U_x^i \rVert_{L^\infty([-5,5])} < e^{\beta t}$ for $i = 1,2$.
    We know that $\theta(p(t), t)$ converges to $0$ and $\theta(q(t),t)$ converges to $\pi$ as $t$ goes to $-\infty$.
    Therefore, for $t\ll -1$,
    \begin{equation}\label{eq: angle difference is pi}
        \theta(p(t), t) = O(e^{\beta t}), \quad \theta(q(t), t) = \pi + O(e^{\beta t}).
    \end{equation}

    By the first variational formula of area and the evolution equation of CSF, 
    \begin{equation*}
        \frac{d}{dt} \lvert F(t)\rvert  = \int_{\pa F(t)} (-\kappa) \, ds = -\int_{p(t)}^{q(t)} \kappa \, ds = -\pi +  O(e^{\beta t}),
    \end{equation*}
    Here, the boundary $\pa F(t)$ is oriented counterclockwise, and the portion on the $x^2$-axis has no contribution to the curvature integral. 
    Thus, choosing $t_0 \ll -1$, as $t\ll t_0$
    \begin{equation*}
        \lvert F(t) \rvert = \lvert F(t_0)\rvert -\int_{t}^{t_0}  \frac{d}{ds} \lvert F(s)\rvert \, ds = \lvert F(t_0)\rvert + \pi t_0 -  \pi t + O(e^{\beta t}),
    \end{equation*}
    and then the desired result follows.

\end{proof}

\subsection{Best-fitting grim reaper soliton}
The right-pointing translating grim reaper soliton, denoted by $\cG_{a_i, a_j;b} = \cup_{t\in \R}\{t\}\times\widehat\Gamma_{a_i, a_j; b}(t)$, bounded between the asymptotes $x^2 = a_i$ and $x^2=a_j$ can be parametrized by
\begin{equation}\label{eq: grim reaper}
    \widehat\Gamma_{a_i, a_j; b}(t):\quad  (x^1, x^2) = \big(\tfrac{1}{k} \log \cos\big(k z\big) - kt + b,\,  z + \bar a\big), \quad \mbox{for }\, -\tfrac{\pi}{2k}< z< \tfrac{\pi}{2k},
\end{equation}
where $\bar a = (a_i + a_j)/2$, $b$ is a \emph{shift constant}, and the positive constant $k = \pi \lvert a_i-a_j\rvert^{-1}$ represents the translation speed of the soliton.
We will suppress the dependence on $a_i, a_j$ in the notation of $\widehat{\Gamma}_b(t)$ for brevity in the following discussion.
Similarly, we denote by $\widehat{F}_b(t)$ the bounded region in the right half-plane enclosed by $\widehat{\Gamma}_b(t)$ and the $x^2$-axis for $t\ll-1$.

\begin{theorem}[best-fitting grim reaper]\label{thm: best.fitting.gr}
    There exists a unique $b\in \R$ such that $\lvert \hat F_b(t)\rvert$ satisfies the same asymptotic expansion for $\lvert F(t)\rvert$ in Proposition~\ref{prop: area asymp}: for $t\ll -1$
    \begin{equation}\label{eq: asymp.expansion.2}
        \lvert \widehat{F}_b(t)\rvert = -\pi t  + C_0 + O(e^{\beta t}).
    \end{equation}
    Moreover,
    \begin{equation*}
        \lim_{t\to -\infty}\lvert F(t) \triangle \widehat{F}_b(t) \rvert =0.
    \end{equation*}
    Here, $A \triangle B =  (A\setminus B)\cup (B\setminus A)$ denotes the symmetric difference of the sets $A$ and $B$.
\end{theorem}

\begin{proof}
    For any $b\in \R$, $c>0$, the area of the unbounded region bounded between $\widehat{\Gamma}_b(t)$ and $\widehat{\Gamma}_{b+c}(t)$ is exactly $\lvert a_i - a_j\rvert c$, which contributes to the increment of the constant term in the expansion \eqref{eq: asymp.expansion.2}. 
    Thus, there exists a unique $b$ so that the constant terms in \eqref{eq: asymp.expansion.1} and \eqref{eq: asymp.expansion.2} match.
    The error estimate $O(e^{\beta t})$ in Proposition \ref{prop: area asymp} applies for the general solutions to CSF that converges to limit lines $x^2=a_i$ and $x^2 = a_j$, including translating grim reapers.  

    Define functions $h(t) := v^1(t)$ and $\hat{h}(t) :=  b-kt$ that represent $x^1$-coordinates of the vertices of $\Gamma(t)$ and $\widehat{\Gamma}_b(t)$, respectively. 
    Since $\hat{h}(t)$ is monotone, we may define a shifting function $\mathcal{E}(t)$ for all $t\ll -1$ such that $h(t) = \hat{h}(t + \mathcal{E}(t))$.
    Let $\vare>0$ be sufficiently small so that $L = 50\vare^{-\frac12}A >2D_0$ where $L$, $A$ and $D_0$ are stated in Theorem \ref{thm: exp.conv.away.tip}. 
    By Proposition~\ref{prop: tip_local_convergence}, there exists $t_{\vare}\ll -1$ such that for all $t\leq T$, $\Gamma(t)$ is $\vare$-close to $\widehat{\Gamma}_b(t+ \cE(t))$ in $C^k$-topology restricted in the domain $\{(x^1, x^2):x^1\geq h(t)-2L\}$.
    There exists a numeric constant $C$ such that for all $t\leq t_{\vare}$
    \begin{equation}\label{eq: area.diff.tip}
        \big\lvert \big(F(t)\triangle \widehat{F}_b(t + \cE(t))\big) \cap \{(x^1, x^2)\in \R^2:x^1\geq h(t)-2L\}\big\rvert  \leq C \vare^{\frac12}.
    \end{equation}
    To see this, we may first decompose the region in \eqref{eq: area.diff.tip} into disjoint subregions enclosed by closed and piecewise smooth curves $C_1, \ldots, C_N$ with counterclockwise orientation that comprise curve segments of $\Gamma_t$ and $\widehat{\Gamma}_b(t + \cE(t))$ and possibly some auxiliary vertical line segments at $x^1 = h(t) - \lvert a_i - a_j\rvert/4$.
    Suppose that $C_1, \ldots, C_M$ are near the tip and can be expressed as unions of graphs over the $x^2$-axis, and $C_{M+1}, \ldots, C_N$ are away from the tip and can be expressed as unions of graphs over the $x^1$-axis and vertical line segments.
    By Green's theorem, the area in \eqref{eq: area.diff.tip} is given by
    \begin{align*}
        \sum_{i=1}^M \int_{C_i} x \, dy + \sum_{i=M+1}^N \int_{C_i} (-y)\, dx \leq 2\vare \lvert a_i - a_j\rvert + 2\vare(2L)  \leq C\vare^{\frac12}.
    \end{align*}
    Furthermore, by Theorem~\ref{thm: exp.conv.away.tip}, there exists $\beta>0$ such that for all $t\leq T$,
    \begin{align*}
        \big\lvert \big(F(t) \triangle \widehat{F}_b(t + \cE(t)) \big) \cap  \{(x^1, x^2)\in \R^2:0\leq x^1\leq h(t)-2L\} \big\rvert \\
        \leq 4\int_{0}^{h(t)-2L} e^{-\beta(h(t) - L -x)}\, dx \leq \tfrac{4}{\beta} e^{-\beta L} < C\vare^{\frac12},
    \end{align*}
    provided that $\vare$ is sufficiently small.
    Putting all together, for all $\vare>0$, there exists $t_{\vare}$ such that for all $t\leq t_{\vare}$,
    \begin{align}\label{eq: area.diff}
        \lvert F(t) \triangle \widehat{F}_b(t + \cE(t))\rvert \leq 2C\vare^{\frac12}.
    \end{align}

    We claim that $\cE(t) \to 0$ as $t\to -\infty$. 
    Using the asymptotic expansion for area \eqref{eq: asymp.expansion.1} and \eqref{eq: asymp.expansion.2} together with \eqref{eq: area.diff}, for all $t\leq t_{\vare}$
    \begin{equation}\label{eq: cE(t)}
        \pi\lvert \cE(t)\rvert + O(e^{\beta t}) =\big\lvert \lvert F(t)\rvert-\lvert \widehat{F}_b(t + \cE(t))\rvert \big\rvert \leq \lvert F(t) \triangle \widehat{F}_b(t + \cE(t))\rvert \leq 2C\vare^{\frac12}.
    \end{equation}
    Letting $\vare\to 0^+$, we complete the claim.

    Recall that $(A\triangle B)\triangle (B\triangle C) = A\triangle C$ for any sets $A,B$, and $C$. 
    Using the observation in the first paragraph and \eqref{eq: cE(t)}, for $t\leq t_{\vare}$
    \begin{align*}
        \lvert F(t) \triangle \widehat{F}_b(t)\rvert &\leq \lvert F(t) \triangle \widehat{F}_b(t+ \cE(t))\rvert + \lvert \widehat{F}_b(t+ \cE(t)) \triangle \widehat{F}_b(t) \rvert\\
        & \leq 2C\vare^{\frac12} + \cE(t)\lvert a_i - a_j\rvert\\
        &\leq 2C \vare^{\frac12} + 2C\vare^{\frac12}\pi^{-1}\lvert a_i - a_j\rvert + O(e^{\beta t}).
    \end{align*}
    Letting $\vare\to 0^+$ yields that $\lvert F(t) \triangle \widehat{F}_b(t)\rvert \to 0$ as $t\to -\infty$.
\end{proof}

\bigskip
\bigskip

\begin{corollary}\label{cor: total translation of tip}
Let $\Gamma_t$ be a right-pointing finger which is asymptotic to the lines $x^2=a_i$ and $x^2=a_j$ as $t \to -\infty$. Then, its tip $v(t)$ satisfies
\begin{equation}\label{eq:tip_limit}
    \lim_{t\to-\infty}\big\{v(t) + \frac{\pi}{|a_i-a_j|} t\be_1\big\}=\big(b, \frac{a_i + a_j}{2}\big)
\end{equation}
where $b \in \mathbb{R}$ is the same as in Theorem~\ref{thm: best.fitting.gr}.
\end{corollary}

\begin{proof}
The position of the tip of the best-fitting grim reaper \eqref{eq: grim reaper} determined in Theorem \ref{thm: best.fitting.gr} is exactly $(-\pi t/\lvert a_i - a_j\rvert + b,\ (a_i + a_j)/2)$.
\end{proof}

\begin{figure}
    \begin{center}
            \includegraphics[width=0.8\linewidth]{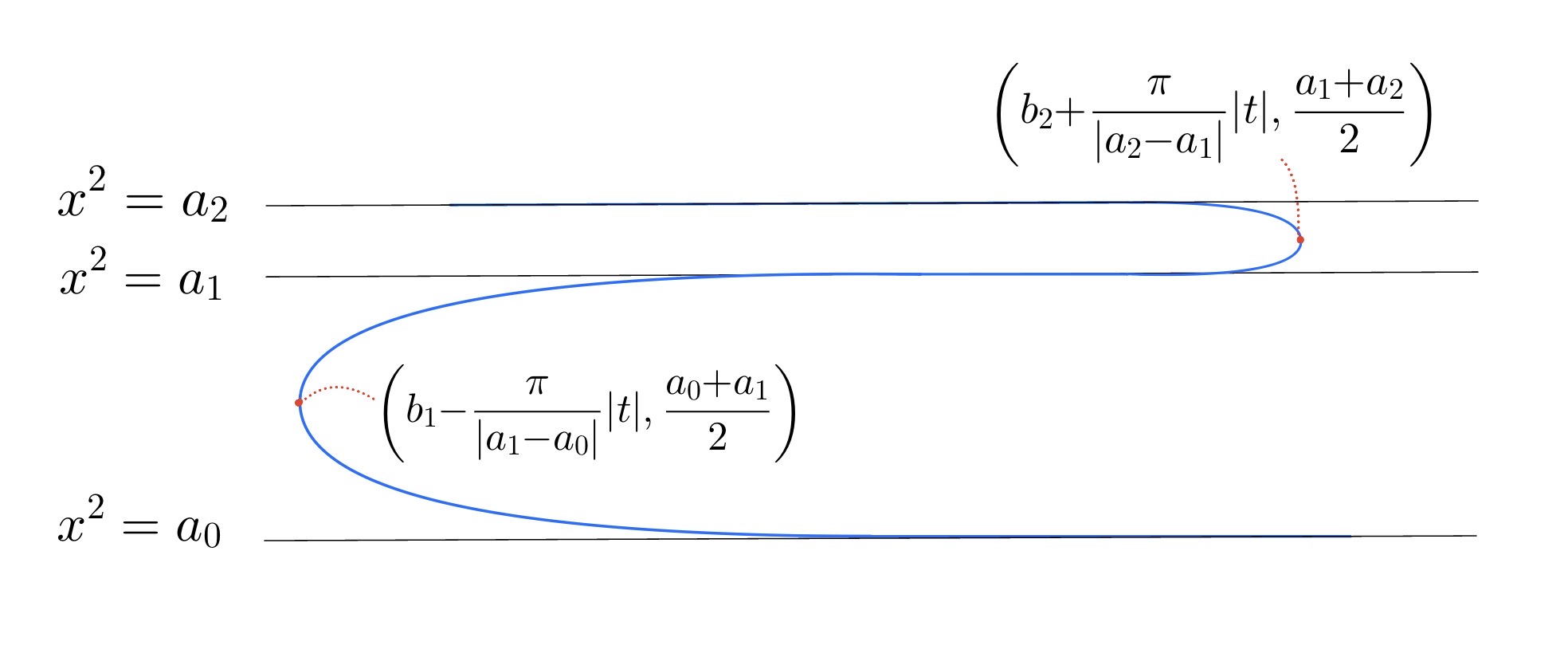}
    \end{center}
    \caption{Tips asymptotic to grim reapers.}
        \label{fig: trombone}
\end{figure}

\bigskip

\section{Classification}
\noindent By Theorem~\ref{thm: sharp asympt line}, the sheets $\Sigma_t^0, \cdots, \Sigma_t^m$ of an ancient solution $\cM$ with $\ent[\cM] =m+1\geq 3$ locally converge to the lines  $x^2 = a_0, \cdots, x^2 = a_m$ as $t \to -\infty$. 
We denote by $\Sigma^0$ the sheet of $M_t$ with lowest $x^2$-intercept and label $\Sigma^i$ according to the given orientation of $\cM$: $\Sigma_t^i$ is followed by $\Sigma_t^{i+1}$ along the orientation for all $i = 0, 1, \ldots, m-1$ if $M_t$ is noncompact, or for all $i = 0, 1, \ldots, m$ and identifying $\Sigma_t^{m+1}$ with $\Sigma_t^0$ if $M_t$ is compact.
For $t \ll -1$, let $\Gamma_t^i$ denote the finger consisting of the sheets $\Sigma_t^i$ and $\Sigma_t^{i+1}$, and let $F^i(t)$ denote the finger region of $\Gamma_t^i$ as defined prior to Proposition \ref{prop: area asymp}.

\begin{theorem}[graphicality]\label{thm: graphicality}
Suppose that $\ent(\mathcal{M}) = m +1 \geq 3$. Then $M_t$ must be noncompact with $a_{0} <\ldots < a_{i-1}< a_{i} < \cdots<a_{m}$. 
Moreover, for all $t\in \mathbb{R}$, $M_t$ is a graph $x^1 = V(x^2,t)$ where $V\in C^\infty\big((a_0, a_m) \times (-\infty, \infty)\big)$. 
See Figure \ref{fig: trombone} for an illustration.
\end{theorem}

\begin{proof}
    Fix a time $t \leq T_{\varepsilon}$ for some sufficiently small $\varepsilon > 0$, where $T_{\varepsilon}$ is defined as in Proposition \ref{prop: tip_local_convergence}. In the following discussion, the statement ``for all $i$" refers to $i = 0, \ldots, m$ in the compact case and $i = 0, \ldots, m-1$ in the non-compact case.
    By nondegeneracy of finger Theorem \ref{thm: non-degenerate finger}, $a_i \neq a_{i+1}$ for all $i$.
    By (\ref{eq: angle difference is pi}), the orientations on the asymptotic lines $x^{2} = a_{i}$ induced by the parametrization of $M_{t}$ are alternating. Consequently, the finger $\Gamma^{i-2}_{t}$ must point in the same direction as the finger $\Gamma^{i}_{t}$ (which is asymptotic to $x^{2} = a_{i}$ and $x^{2} = a_{i+1}$). In particular, all fingers $\Gamma_t^i$ with an odd index $i$ point in one direction, while all fingers with an even index $i$ point in the opposite direction.
    
    Note that by the embeddedness of $M_t$, if $i \equiv j \pmod 2$, the finger regions must satisfy a nesting or disjointness property. Specifically, the regions are either nested, such that $F^i(t) \subseteq F^j(t)$ or $F^j(t) \subseteq F^i(t)$, or they are disjoint, such that $F^i(t) \cap F^j(t) = \emptyset$. Notably, these two cases are equivalent to the nesting or disjointness of the \emph{open} intervals $I_i$ and $I_j$ formed by the asymptotic values $\{a_i, a_{i+1}\}$ and $\{a_j, a_{j+1}\}$, respectively.

    Observe that \emph{proper} nesting of the open intervals $I_i$ and $I_j$ is impossible for $i \equiv j \pmod 2$. Otherwise, by Corollary \ref{cor: total translation of tip}, the travel speed of one finger would be strictly greater than that of the other. This discrepancy in speed implies that at some time $\bar{t} \leq t$, the fingers must collide, such that $\Gamma^{i}_{\bar{t}} \cap \Gamma^{j}_{\bar{t}} \neq \emptyset$, contradicting embeddedness.

    We claim that $a_i < a_{i+1}$ for all $i$.
    It immediately follows that $M_t$ cannot be compact since $a_0 < a_{m} < a_{m+1}= a_0$, a contradiction.
    Suppose the claim is false, given that $\Sigma_t^{0}$ is the lowest sheet, there exists a smallest index $i \geq 1$ such that $a_{i-1} < a_i$ but $a_{i+1} < a_i$.
    We will examine two exhaustive cases $a_{i+1} < a_{i-1} < a_i$ and $a_{i-1}\leq a_{i+1} < a_i$, separately.

    \

    \noindent\textbf{Case 1} ($a_{i+1} < a_{i-1} < a_{i}$): Since $a_{0}$ is the minimal value, it follows that $i \geq 2$. 
    By the nesting or disjoint property of fingers, $(a_{i-2}, a_{i-1})\subsetneq (a_{i+1}, a_i)$. However, this nesting is proper, which is impossible. See Figure~\ref{fig: impos.trom.1} for an illustration.

    \begin{figure}[h]
    \begin{center}
            \includegraphics[width=0.85\linewidth]{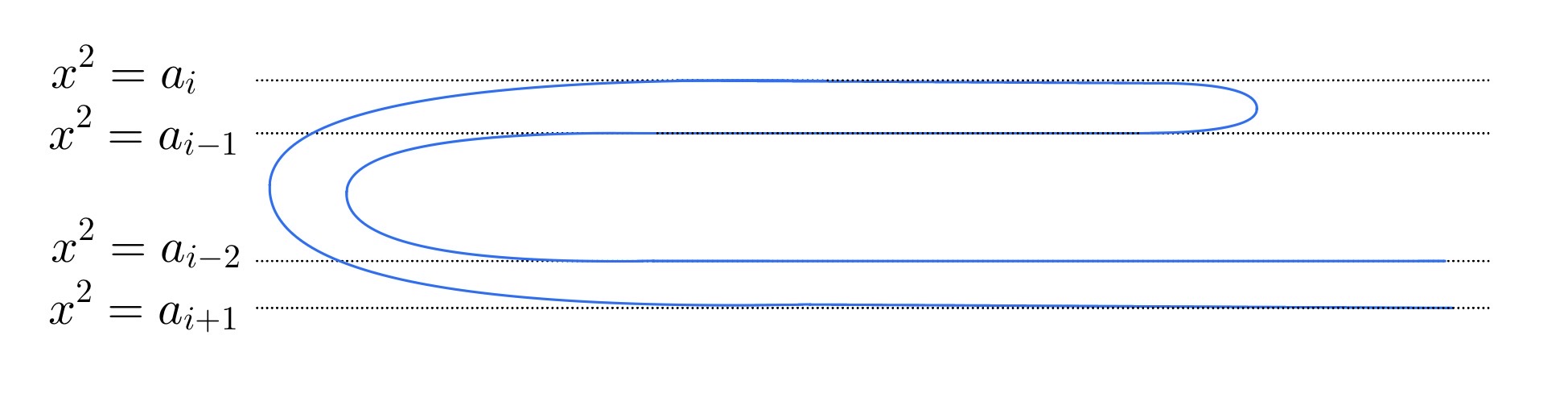}
    \end{center}
    \caption{Case 1: $a_{i+1} < a_{i-1} < a_{i}$: proper nesting of fingers.}
        \label{fig: impos.trom.1}
    \end{figure}

    \noindent\textbf{Case 2} ($a_{i-1} \leq a_{i+1} < a_i$): For the subsequent discussion, let $q_k = (0, q_k^2)$ denote the $x^2$-intercept of $\Sigma_t^k$ for each $k$. By the embeddedness of $M_t$ and the assumption $\text{Ent}(\mathcal{M}) \geq 3$, we have $q_{i-1}^2 < q_{i+1}^2 < q_i^2$, since the case $q_{i-1}^2 = q_{i+1}^2$ can only occur when $i=m=1$, which implies $\text{Ent}(\mathcal{M})=2$.

    The sheet $\Sigma^{i+1}_{t}$ cannot contain a tail; otherwise, the finite distance from $q_{i+1}$ to the tip of the finger $\Gamma^{i-1}_{t}$ necessitates a self-intersection point, a contradiction. Thus, $\Sigma^{i+1}_{t}$ is connected to a finger $\Gamma^{i+1}_{t}$, implying $a_{i+2}$ exists and $m \geq i+2$. 

    \begin{figure}[h]
    \begin{center}
            \includegraphics[width=0.75\linewidth]{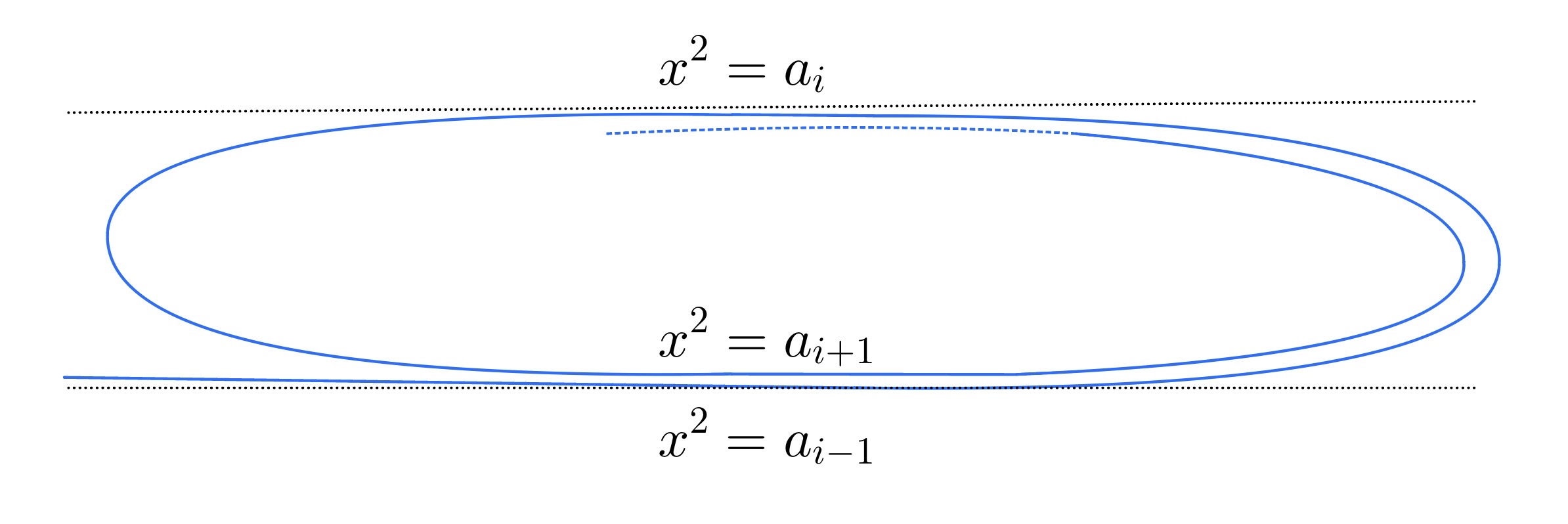}
    \end{center}
    \caption{Case 2: $a_{i-1} \leq a_{i+1} < a_i$: spiral and nesting chains.}
        \label{fig: impos.trom.2}
    \end{figure}

    Next, we must have $a_{i-1} = a_{i+1}$. Otherwise, by the nesting or disjointness property of fingers, $(a_{i+1}, a_{i+2}) \subsetneq (a_{i-1}, a_i)$, which is impossible. Inductively, $\Sigma_t^m$ does not contain a tail and, to avoid the proper nesting of open intervals $I_j$ of the same parity, we must have $a_{j} = a_{i-1}$ for all $j \geq i-1$ with $j \equiv i-1 \pmod 2$, and $a_j = a_i$ for all $j \geq i$ with $j \equiv i \pmod 2$.

    This configuration implies two \emph{finite} descending nesting chains of finger regions:
    \begin{equation*}
        F^{i-1}(t) \supset F^{i+1}(t) \supset \cdots \quad \text{and} \quad F^{i}(t) \supset F^{i+2}(t) \supset \cdots.
    \end{equation*}
    But this is impossible. Specifically, $M_t$ is compact (since $\Sigma_t^m$ does not contain a tail and, by \cite[Proposition 6.14]{CSSZ24}, tails must occur in pairs) and $\Gamma_t^{m}$ must connect back to $q_0$. However, $q_0$ lies outside both $F^{i+1}(t)$ and $F^i(t)$; the nested structure of these finger regions thus prevents $\Gamma_t^{m}$ from reaching $q_0$ without intersecting earlier segments of the manifold. This contradiction completes the proof of the claim. See Figure~\ref{fig: impos.trom.2} for an illustration.

    We have thus shown that $M_{t}$ is complete non-compact and the asymptotic constants satisfy $a_{0} < \dots < a_{m}$. The graphicality of $M_{t}$ then follows from Proposition~\ref{prop: min angle at infl pt} for the sheets connecting fingers and Proposition~\ref{prop: anal.ang.kap.} {\bf (B)} for tails.
    
\end{proof}

\bigskip

Let us summarize the progress made thus far:
\begin{itemize}
    \item \textbf{Asymptotic Lines:} According to Theorem \ref{thm: sharp asympt line}, the ancient solution $\cM$ with $\ent[\cM] =m+1\geq 3$ locally converges to the lines  $x^2 = a_0, \cdots, x^2 = a_m$.
    
    \item \textbf{Finger Convergence:} Following Proposition \ref{prop: tip_local_convergence} and Theorem \ref{thm: graphicality}, $M_t=\{(V(z,t),z):z \in (a_0,a_m)\}$ is a non-compact complete graph with $m$ fingers converging to translating grim reapers, where each finger $\Gamma_t^i$ with $i \in \{1,\cdots,m\}$ is asymptotic to the lines $x^2 = a_{i-1}$ and $x^2 = a_i$. Moreover, two adjacent translating grim reapers have different orientations alternating between right-pointing and left-pointing.
    
    \item \textbf{Horizontal Parameters:} Finally, by Corollary \ref{cor: total translation of tip} for each finger $\Gamma^i_t$ we can find the best-fitting grim reaper solitons $\widehat{\Gamma}_t^i$ determined by the horizontal shift $b_i$.
\end{itemize}

By prescribing the same sequence of best-fitting grim reapers $\widehat{\Gamma}_t^i$ determined by the sequences $\mathbf{a} = \{a_0, \ldots, a_m\}$ and $\mathbf{b}=\{b_1, \ldots, b_m\}$ with orientations determined by (only one of) tails of $\cM$, Angenent and You constructed a \emph{graphical ancient trombone solution} $\cT_{\mathbf{a}, \mathbf{b}}^{\pm}$ in Section 3 of \cite{AY21}. 
We will represent this solution as the graph $x^1 = \bar{V}(x^2, t)$ of the function $\bar{V} \in C^\infty\big( (a_0, a_m) \times (-\infty, \infty) \big)$ in what follows.

\begin{proposition}\label{prop: area.conv}
    Let $V$ and $\bar{V}$ be profile functions of  $\cM$ and the best-fitting trombone $\cT = \cup_t\, T_t\times \{t\}$, respectively. Then, $\int_{a_0}^{a_m} \lvert V(z,t)-\bar  V(z,t)\rvert\,dz$ is finite for all $t\ll -1$ and
    \begin{equation}\label{eq: area.conv}
        \lim_{t\to -\infty}  \int_{a_0}^{a_m} \lvert V(z,t)-\bar  V(z,t)\rvert\,dz= 0.
    \end{equation}
\end{proposition}

\begin{proof}
Let $F^i(t)$ and $\widehat{F}^i(t)$ denote the finger regions of $\Gamma^i_t$ and the best-fitting grim reaper $\widehat{\Gamma}^i_t$ respectively for $i = 1,\ldots, m$ as before; and let $\bar{F}^i(t)$ denote the corresponding finger regions of $\cT$ that are asymptotic to the same best-fitting grim reaper $\widehat{\Gamma}_t^i$ for $i = 1, \ldots, m$.
Let $E^j(t)$ and $\bar{E}^j(t)$ denote the tail regions enclosed by the $x^2$-axis, horizontal line $x^2 = a_j$, and the nearby tails of $M_t$ and $\bar{M}_t$, respectively, for $j = 0$ or $m$.

The integral in \eqref{eq: area.conv} represents the sum of areas bounded between $M_t$ and $\bar{M}_t$.
Therefore, using the triangle inequality for symmetric difference,
\begin{align}
    \int_{a_0}^{a_m} &\lvert V(z,t) - \bar{V}(z,t)\rvert\, dz = \lvert E^0(t) \triangle \bar{E}^0(t)\rvert + \lvert E^m(t) \triangle \bar{E}^m(t)\rvert + \sum_{i = 1}^m \lvert F^i(t) \triangle \bar{F}^i(t)\rvert \notag\\
    &\leq \lvert E^0(t) \triangle \bar{E}^0(t)\rvert + \lvert E^m(t) \triangle \bar{E}^m(t)\rvert+ \sum_{i = 1}^m \lvert F^i(t) \triangle \widehat{F}^i(t)\rvert +\sum_{i = 1}^m \lvert \widehat{F}^i(t) \triangle \bar{F}^i(t)\rvert.\label{eq: area.conv.1}
\end{align}
According to Theorem \ref{thm: exp.conv.away.tip} Case (B), there exists (smaller) $\beta>0$ so that the tails of $M_t$ and $T_t$ can be represented by graphs $x^2 = O\big(\exp(\,\beta(t -\lvert x^1\rvert)\,)\big)$ for $t\ll -1$, and therefore, the first two terms in \eqref{eq: area.conv.1} are finite and converge to $0$ as $t$ goes to $-\infty$.
Applying Theorem \ref{thm: best.fitting.gr} to each finger of the general solution $\cM$ and trombone $\cT$, all terms in the summations in \eqref{eq: area.conv} also converge to 0 as $t$ goes to $-\infty$. 
This completes the proof.
\end{proof}

\begin{theorem}[classification]\label{thm: classification}
Any ancient, embedded, smooth solution $\mathcal{M}$ to the curve shortening flow with entropy $\ent[\cM] = m+1 \geq 3$ is a graphical ancient Angenent--You trombone solution. Furthermore, $\mathcal{M}$ is uniquely characterized by its heights $\mathbf{a}=\{a_0, \ldots, a_m\}$, horizontal shifts $\mathbf{b}=\{b_1, \ldots, b_m\}$, and the direction of its tails.
\end{theorem}

\begin{proof}
Let $V$ and $\bar{V}$ be profile functions of $\cM$ and the best-fitting trombone $\cT = \cup_t\, T_t\times \{t\}$, respectively, as before, which are solutions to graphical CSF:
\begin{equation*}
    \pa_t V = \frac{V_{zz}}{1+ V_{z}^2} = (\arctan V_z)_z, \quad \pa_t \bar{V} = \frac{\bar{V} _{zz}}{1+ \bar{V} _{z}^2} = (\arctan \bar{V} _z)_z.
\end{equation*}
Define the function $W=V-\bar{V}$. 
Then, $W$ satisfies
\begin{equation}\label{eq: evol.W}
    \pa_t W = \partial_z (\arctan(V_z)-\arctan(\bar V_z)).
\end{equation}
For simplicity, we write $\varphi(z,t) = \arctan(V_z)$ and $\bar{\varphi}(z,t) = \arctan(\bar{V}_z)$.

For all $t\ll -1$, define the integral $A(t) := \int_{a_0}^{a_m} \lvert W(z,t) \rvert \, dz$ that represents the total area of all regions bounded between $M_t$ and $T_t$.
By Proposition \ref{prop: area.conv}, $A(t)$ is finite for all $t\ll -1$ and converges to 0 as $t\to -\infty$.
We claim that for all $t_1 \leq t_2\ll -1$, we have
\begin{equation}\label{eq: dec.A}
    A(t_1) \geq A(t_2).
\end{equation}
For $\vare, \delta>0$ sufficiently small, consider the integral $A_{\vare, \delta}(t):=\int_{a_0+\delta}^{a_m-\delta} (W^2+ \vare^2)^{\frac12}\, dz$. 
By dominant convergence theorem, $A_{\vare,\delta}(t)$ converges to $A(t)$ for all $t\ll -1$ as $\vare \to 0^+$, $\delta \to 0^+$.
In light of \eqref{eq: evol.W}, integrating by parts yields
\begin{align}
    \frac{d}{dt}A_{\vare, \delta}(t) &= \int_{a_0+\delta}^{a_m-\delta} \frac{W}{\sqrt{W^2+\varepsilon}} W_t\, dz\label{eq: dA.1}\\
    &= \frac{W}{\sqrt{W^2+\varepsilon}}\big(\varphi-\bar\varphi\big)\Big\rvert_{a_0+\delta}^{a_m - \delta}
     -\int_{a_0+\delta}^{a_m-\delta} \frac{\varepsilon W_z}{(W^2+\varepsilon)^{\frac{3}{2}}}\big(\varphi-\bar\varphi\big)\, dz.\notag
\end{align}
Observe that since $\tan \theta$ is increasing on $(-\pi/2, \pi/2)$, for any $\alpha, \beta\in (-\pi/2, \pi/2)$
\begin{equation*}
    (\tan \alpha- \tan \beta)(\alpha - \beta) \geq 0.
\end{equation*}
Using this observation, we find
\begin{equation*}
    -\int_{a_0+\delta}^{a_m-\delta} \frac{\varepsilon W_z}{(W^2+\varepsilon)^{\frac{3}{2}}}\big(\varphi-\bar\varphi\big)\, dz = -\int_{a_0+\delta}^{a_m-\delta} \frac{\varepsilon (\tan\varphi - \tan \bar\varphi)}{(W^2+\varepsilon)^{\frac{3}{2}}}\big(\varphi-\bar\varphi\big)\, dz \leq 0.
\end{equation*}
Therefore, integrating \eqref{eq: dA.1} over $[t_1, t_2]$ gives
\begin{align}\label{eq: D.int.W}
    A_{\vare, \delta}(t_2) - A_{\vare, \delta}(t_1)\leq  \int_{t_1}^{t_2 }\frac{W}{\sqrt{W^2+\varepsilon}}\big(\varphi-\bar\varphi\big)\Big\rvert_{a_0+\delta}^{a_m - \delta} \, dt.
\end{align}
Note that the RHS of \eqref{eq: D.int.W} converges to $0$ as $\delta\to 0^+$ since $\lvert W\rvert/\sqrt{W^2+ \vare^2}$ is bounded and $\varphi - \bar\varphi$ converges to $0$ as $z$ goes to $a_0$ and $a_m$ for all $t\ll -1$ by Proposition \ref{prop: anal.ang.kap.} Case (B).
Finally, letting $\delta \to 0^+$ and then $\vare\to 0^+$ completes the proof of the claim.

By \eqref{eq: dec.A} and Proposition \ref{prop: area.conv}, for all $t\ll -1$
\begin{equation*}
    0\leq  A(t) \leq \lim_{s\to -\infty}A(s) = 0.
\end{equation*}
It follows that $V\equiv \bar V$ for all $t\ll -1$, and for all $t$ by the identity principle.
\end{proof}

\bigskip

\begin{proof}[Proof of Theorem \ref{thm:main}]
If $\ent[\cM] < 3$, then by \cite[Theorem~1.5]{CSSZ24}, $\cM$ is a static line, a shrinking circle, a paper clip, or a translating grim reaper. If $3 \leq \ent[\cM] < \infty$, then by Theorem~\ref{thm: classification}, $\cM$ is a graphical ancient trombone solution.
\end{proof}

\begin{proof}[Proof of Corollary~\ref{cor:main1}]
By \cite[Theorem~1.1]{SZ24}, for an ancient, complete, smooth, embedded curve shortening flow, having finite total curvature is equivalent to having finite entropy. The result then follows from Theorem~\ref{thm:main}.
\end{proof}

\begin{proof}[Proof of Corollary~\ref{cor:main2}]
Note that the only closed solutions in the classification given by Theorem~\ref{thm:main} are the shrinking circle and the paper clip; notably, both of these solutions are convex.
\end{proof}

\begin{proof}[Proof of Corollary~\ref{cor:main3}]
Observe that the non-static, non-compact solutions in the classification provided in Theorem~\ref{thm:main} are the translating grim reaper and the ancient trombone; both are ancient solutions that evolve as complete graphs over a bounded open interval.
\end{proof}

\begin{proof}[Proof of Corollary~\ref{cor:main4}]
The result then follows directly from Theorem~\ref{thm: classification}.
\end{proof}

\bigskip
\bigskip

\subsection*{Acknowledgments}
KC has been supported by the KIAS Individual Grant MG078902, an Asian Young Scientist Fellowship, and the National Research Foundation(NRF) grants funded by the Korea government(MSIT) (RS-2023-00219980) and (RS2024-00345403); DS was partially supported by the PRIN project 20225J97H5 funded by Ministero dell'Università e della Ricerca in Italia and the grant no. EUR2024-153556 funded by MICIU/AEI/10.13039/501100011033; WS is supported by the Taiwan NSTC grant 114-2115-M-008-012-MY3.

 \bibliographystyle{amsplain} 
\bibliography{CSF}
\end{document}